\documentclass[12pt, a4paper]{article}

\usepackage{mathptmx}       % selects Times Roman as basic font
\usepackage{helvet}         % selects Helvetica as sans-serif font
\usepackage{courier}        % selects Courier as typewriter font
\usepackage{makeidx}         % allows index generation
\usepackage{graphicx}        % standard LaTeX graphics tool
\usepackage{multicol}        % used for the two-column index
\usepackage[bottom]{footmisc}% places footnotes at page bottom

\usepackage{amsmath,amsfonts,amsthm}
\usepackage[a4paper,left=3cm,right=2cm,top=2.5cm,bottom=2.5cm]{geometry}

\makeindex             % used for the subject index
\newcommand{\R}{\mathbb{R}}

\newtheorem{theorem}{Theorem}[section]

\newtheorem{proposition}{Proposition}[theorem]

\begin{document}

\date{\today}

\title{High order Nystr\"om methods for transmission problems for Helmholtz equation}

 \author{V\'{\i}ctor Dom\'{\i}nguez\footnote{Departamento de Ingenier\'{\i}a Matem\'{a}tica e Inform\'{a}tica, Universidad P\'{u}blica de Navarra. Avda
  Tarazona s/n, 31500 Tudela, Spain.  Email: {\tt victor.dominguez@unavarra.es}} \and Catalin Turc\footnote{Department of Mathematical Sciences and Center for Applied Mathematics, NJIT, Univ. Heights. 323
  Dr. M.L. King Jr. Blvd, Newark, NJ 07102, USA. Email: {\tt catalin.c.turc@njit.edu}}}

% \title*{High order Nystr\"om methods for transmission problems for Helmholtz Equation}
% % Use \titlerunning{Short Title} for an abbreviated version of
% % your contribution title if the original one is too long
% \author{V\'{\i}ctor Dom\'{\i}nguez and Catalin Turc}
% % Use \authorrunning{Short Title} for an abbreviated version of
% % your contribution title if the original one is too long
% \institute{Departamento de Ingenier\'{\i}a Matem\'{a}tica e Inform\'{a}tica, Universidad P\'{u}blica de Navarra. Avda
% Tarazona s/n, 31500 Tudela, Spain. \email{victor.dominguez@unavarra.es} 
% \and Department of Mathematical Sciences and Center for Applied Mathematics, NJIT, Univ. Heights. 323
% Dr. M.L. King Jr. Blvd, Newark, NJ 07102, USA. \email{catalin.c.turc@njit.edu}}
\maketitle

\abstract{We present superalgebraic compatible Nystr\"om discretizations for the four Helmholtz boundary operators of Calder\'{o}n's calculus on smooth closed curves in 2D. These discretizations are based on appropriate  splitting   of the kernels  combined with very accurate product-quadrature rules for the different singularities that such kernels present. A Fourier based analysis shows that the four discrete operators converge to the continuous ones in appropriate Sobolev norms. This proves that Nystr\"om discretizations of many popular integral equation formulations for Helmholtz equations are stable and convergent. The convergence is actually superalgebraic for smooth solutions.}

% 
% 
% In this work we present high-order compatible Nystr\"om discretizations for  
% Helmholtz boundary layer operators in 2D.  These discretizations are based on accurate product-quadrature rules for evaluating the underlying integral operators. The  numerical errors of these discretizations  are analyzed in Sobolev norms. We then show how these approximations can be used to solve numerically    several boundary integral formulations for Helmholtz transmission problems and show that the convergence of these scheme is superalgebraic for smooth data. Numerical experiments are presented  to support the theoretical results.}

\section{Introduction}

The design of robust discretizations of the boundary integral equations in 2D has been an active research topic in the last decades.  The analysis of Galerkin discretizations of boundary integral equations is by now well understood in the case of smooth boundaries and boundary data. Indeed, their stability can be established based on the coercivity of the principal parts of the boundary integral operators featured in the integral formulations (a first result along these lines can be traced back to \cite{Nedelec02}), and compact perturbation analysis arguments. On the other hand, although Nystr\"om/collocation methods are simpler to implement, their analysis is somewhat more complicated. Given that for  2D  problems boundary integral operators can be thought of as periodic pseudodifferential operators, the analysis of discretization schemes for boundary integral equations relies on Fourier analysis. Galerkin as well as Nystr\"om/collocation methods for periodic integral equations have been fully analyzed for many periodic integral equations and these techniques have been also used to  derive new methods as 
 qualocation schemes,  cf. \cite{saranen_vainiko} and references therein.

% The design of high-order methods  for  boundary element equations on smooth curves has been considered in the literature for many years. We can mention for instance Galerkin, collocations or
% qualocation schemes,  cf. \cite{saranen_vainiko} and references therein, which  provide fast convergence schemes for integral equations related with Helmholtz and other differential equations in 2D. All of these schemes make often use of trigonometric polynomials as discrete spaces in which to project the solutions of the underlying boundary integral equations. 

Boundary integral formulations of Helmholtz equations in a certain domain rely on single and double layer acoustic potentials and their Dirichlet and Neumann traces on the boundary of that domain. These traces lead to the natural definition of four boundary integral operators which are referred to as the Helmholtz boundary integral operators of Calderon's calculus. In this paper we focus on Nystr\"om methods based on suitable quadrature rules for the discretization  of the four Helmholtz boundary integral operators that feature in Calderon's calculus. These provide a means of defining fully discrete versions of these operators which can be used easily to discretize complicate formulations involving rather complex compositions of different boundary operators. Moreover, these discretizations can be easily used in conjunction with iterative solvers based on Krylov subspace methods. 

% 
% fulfill some good properties which make them one of the algorithms to be considered in this frame. On one hand, they provide high order methods without demanding large computational cost. On the other hand, these schemes are naturally adapted for iterative solver solutions of the linear systems that they produce. 

The aim of this paper is not to propose new discretizations of the Helmholtz boundary integral operators. Actually, most of those considered here can be found and have been thoroughly analyzed in the literature, mostly by Kress (cf \cite{kressH, KressLI} and references therein). Our objective  is therefore different: we want to propose {\em compatible} discretizations of the  four Helmholtz boundary integral operators that lead to superalgebraic schemes for most of the boundary integral formulations of the Helmholtz equation in 2D. 

Helmholtz transmission problems for smooth interfaces provide a sufficiently complex environment for testing our discretizations as they feature all of the four Helmholtz boundary integral operators in Calderon's calculus. Discretizations of integral formulations of other types of boundary conditions can be readily produced and analyzed with the methods we present in this paper.

 Some of the formulations considered in this paper are direct, i.e. the unknowns are physical quantities of the problem (typically the trace and the normal derivative of the solution), others are indirect. Some of the indirect formulations considered in this text could be more economical from a computational point of view. Besides, some more sophisticated integral formulations lead to  matrices with clustered eigenvalues, which usually ensures a faster convergence of Krylov methods such as GMRES. Demanding better spectral properties requires working with more complex formulations whose discretization could seem challenging at first sight. We will show that the discrete boundary layer operators can be used as black boxes in such a way that the discretization of any integral formulation, however complicated, is in fact straightforward. Moreover, for smooth data, we prove that the numerical solutions  converge superalgebraically, that is, faster than any negative power of $N$, the number of degrees of freedom.

The paper is structured as follows: in Section 2 we discuss briefly the Helmholtz transmission problem and introduce the boundary layer potentials and operators for the Helmholtz equation. In Section 3 we reformulate these mappings as integral operators acting on spaces of $2\pi-$periodic functions via a parameterization of the interface. {We present also  their numerical discretizations and analyze their convergence. We next introduce compatible discretizations of the operators and derive convergence estimates in Sobolev setting}. We conclude by showing in Section 4 how these  compatible discretizations can be applied to solve numerically several boundary integral formulations of the original Helmholtz transmission problem. Well-posedness and convergence estimates are derived for the integral equations considered in this paper. Some numerical experiments are presented {in  the final Section 5}.

\section{Helmholtz transmission problems and boundary integral operators}

We start introducing the domain of the transmission problem (see Figure \ref{fig:domain}).
Let $D^-$ be a compact domain with smooth boundary $\Gamma$ which for simplicity we will assume to be simply connected. Denote also $D^{+}:=\mathbb{R}^2\setminus \overline{D^-}$. We will write $\gamma$ for the trace operator and  $\partial_n$ for the unit normal derivative on $\Gamma$ pointing toward $D^{+}$. Given two wavenumbers $k_+, k_-$ that are complex numbers with non-negative imaginary part, we consider the following Helmholtz transmission problem:

\begin{figure}[t]
\centerline{
 \includegraphics[width=0.6\textwidth]{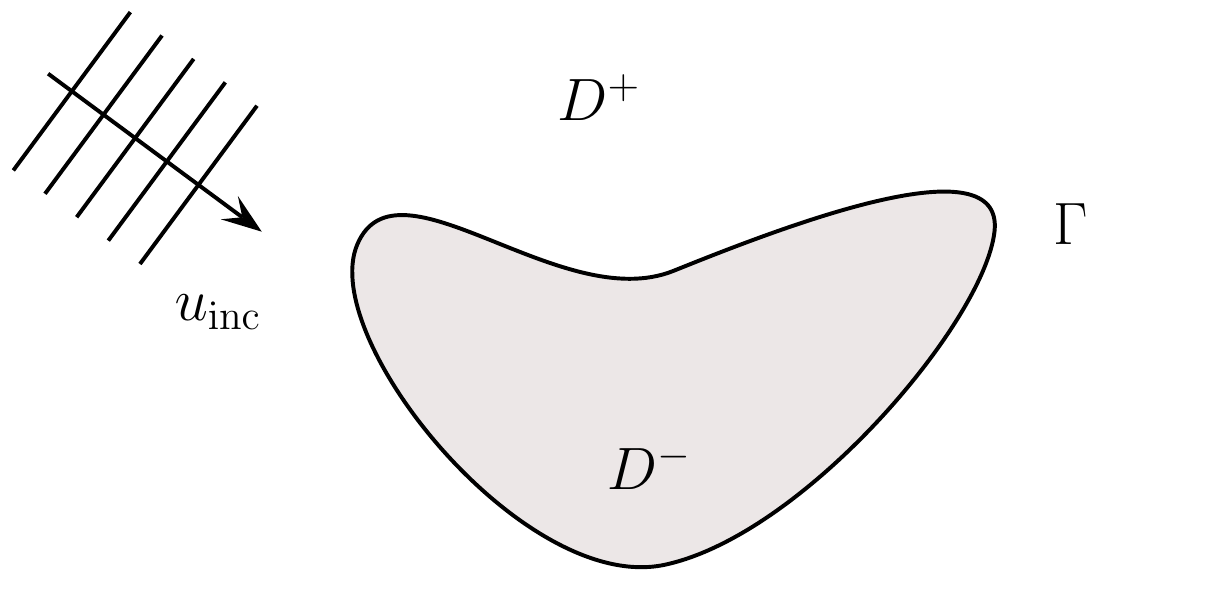}}
 \caption{\label{fig:domain}Sketch of the domain of the transmission problem}
\end{figure}
 
\begin{equation}\label{eq:Tr}
 \begin{array}{rcl}
  \displaystyle \Delta u^{\rm +}+k_{+}^2u^{\rm +}\!\!\!&=&\!\!\!0 \qquad \mathrm{in}\ D^{+}
\\
  \displaystyle   \Delta u^-+ k_-^2 u^-\!\!\!&=&\!\!\!0\qquad \mathrm{in}\ D^{-}\\
   \displaystyle {\gamma} u^{\rm +}\!- {\gamma} u^-\!\!\!&=&\!\!\!-{ u}_{\rm inc}, \  
  \displaystyle {\partial_n} u^{\rm +} \!-\!
  \displaystyle  {\nu}{ {\partial_n}} u^-=
-\partial_n u_{\rm inc} \\
  \displaystyle  \partial_r u^{+}\!-\!{\rm i}k_+ u^{+}\!\!\!&=&\!\!\!o(|{\bf r}|^{-1/2}). 
\end{array}
\end{equation}
Here  $\partial_r$ is the partial derivative on the radial direction and $u_{\rm inc}$ is an incident wave that is a solution of the 
Helmholtz problem for $k_+$ on a neighborhood of $\overline{D^-}$. We assume that the transmission problem above together with its adjoint, that is the transmission problem defined by taking  $k_\pm$ in $D^{\mp}$, are uniquely solvable. For instance, if $k_\pm$ are real and $\nu>0$ these hypotheses are known to be satisfied. We refer to \cite{KlMa} for more comprehensive sets of values of $k_\pm$ and $\nu$ fulfilling these hypotheses. 

%We present briefly in what follows several boundary integral equation formulations of Helmholtz transmission problems. 
Let 
\[
 \Phi_k({\bf x}):=\frac{\rm i}{4} H_0^{(1)}(k|{\bf x}|)
\]
($H_0^{(1)}$ is the Hankel function of first kind and order 0) 
be the outgoing fundamental solution of the Helmholtz equation in $\mathbb{R}^2$. The single and double layer operators are defined as follows
\begin{equation}\label{eq:pot}
%\begin{array}{rcl}
 {SL}_k \varphi % &:=&\displaystyle 
 := \int_\Gamma \Phi_k(\:\cdot\:-\mathbf{y})\varphi(\mathbf{y}){\rm d}\sigma(\mathbf{y}),
 %\\
 \qquad
 { DL}_k g:= 
 %&:=& \displaystyle  
 \int_\Gamma \frac{\partial \Phi_k(\:\cdot\:-\mathbf{y})}{\partial\mathbf{n}(\mathbf{y})}g(\mathbf{y}){\rm d}\sigma(\mathbf{y}).
%\end{array}
\end{equation}
We stress that for any density, the layer operators define solutions of the Helmholtz equation in $\mathbb{R}^2\setminus\Gamma$ which satisfy, in addition, the radiation condition at infinity (last condition in \eqref{eq:Tr}). Moreover, the third Green formula states 
\begin{equation}\label{eq:green}
 u^{\pm}=\mp \: { SL}_{k_{\pm}}\partial_n u_{\pm}\:\: \pm \:\:{DL}_{k_{\pm}}\:\gamma u_{\pm}.
\end{equation}
Let us denote by $\gamma^\pm, \partial_n^\pm$ the trace and respectively the normal derivative taken from $D^{\pm}$. We have then the jump properties 
\begin{equation}\label{eq:bl}
\begin{aligned}
 \gamma^\pm  { SL}_k &=V_k &
 \partial_n^\pm  { SL}_k &=\mp \tfrac12 I+K^\top_k\\
 \gamma^\pm  { DL}_k &=\pm \tfrac12 I+K^\top_k &
 \partial_n^\pm  { DL}_k &= H_k
\end{aligned}
\end{equation}
where $I$ denotes the identity, $V_k$ is the single layer operator, 
$K_k$ and $K_k^\top$ are the double layer and 
adjoint double layer operator, and $H_k$ is the hypersingular operator. 

We can now proceed as follows: (a) we can use \eqref{eq:bl} and the transmission conditions
stated in \eqref{eq:Tr} to compute the Cauchy data of the solution $(u^\pm,
\partial_n^\pm u)$ and reconstruct these functions using~\eqref{eq:green}; (b) we can try to write the $u^\pm$ in terms of some unknown densities associated with the potentials \eqref{eq:pot} and solve for these densities via  equations obtained from \eqref{eq:bl}. Approach (a) leads to  the so-called direct methods whereas schemes obtained from (b) are known as indirect methods.

\section{Associated periodic integral operators and their approximation}

\subsection{Periodic integral operators}

Let us consider a smooth regular $2\pi-$periodic parameterization of the curve $\Gamma$ given by ${\bf x}:\R\to\Gamma$.  First we set the transmission data
\begin{equation}\label{eq:rhs}
%\begin{array}{rcl}
h(s):=-\big(\gamma_\Gamma u_{\rm inc}\circ{\bf x}\big)(s),\quad 
\eta(s):=-\big(\partial_n u_{\rm inc}\circ{\bf x}\big)(s)\,|{\bf x}'(s)|. 
%\end{array}
\end{equation}
We follow the same rule to reformulate layer potentials and boundary integral operators as $2\pi-$periodic integral operators in the following sense: for ${SL}_k$ and the associated boundary integral operators $V_k$ and $K_k^\top$, the norm of the parameterization $|{\bf x}'(t)|$ ({$t$ is the integration variable}) is incorporated  in the density function $\varphi$ in \eqref{eq:pot}, whereas
for ${DL}_k$, and the corresponding boundary integral operators $K_k$ and $H_k$ this term is incorporated in the kernels of these operators.  In addition, the operators $K_k^\top$ and $H_k$  are  multiplied by $|{\bf x}(s)|$, where $s$ {will be used  henceforth} as the variable corresponding to the target point in all of the integral operators considered in this text. With these conventions, we write the single,   double and  adjoint double layer operator as follows 
\begin{eqnarray}
{[}{\rm V}_k\varphi{]}(s)&=&\int_0^{2\pi} A(s,t)\log\sin^2\tfrac{s-t}{2}\varphi(t){\rm d}t
%\nonumber  \\ 
\label{eq:Vk}
%&&
+\int_0^{2\pi} B(s,t)\varphi(t){\rm d}t\\
{[}{\rm K}_k g{]}(s)&=&\int_0^{2\pi} C(s,t)\sin^2\tfrac{s-t}{2}\log\sin^2\tfrac{s-t}{2}g(t){\rm d}t
%\nonumber\\
+\int_0^{2\pi} D(s,t)g(t){\rm d}t
\label{eq:Kk}
\\
{[}{\rm K}^\top_k \varphi{]}(s)&=&\int_0^{2\pi} C(t,s)\sin^2\tfrac{t-s}{2}\log\sin^2\tfrac{t-s}{2}\varphi(t){\rm d}t
%\nonumber\\
\label{eq:Kkt}
%&&
+\int_0^{2\pi} D(t,s)\varphi(t){\rm d}t.
\end{eqnarray}
with
\begin{eqnarray*}
A(s,t)&=&-\frac{1}{4\pi}
J_0(k|{\bf x}(s)-{\bf x}(t)|)\\
C(s,t)&=&-\frac{k ({\bf x}(s)-{\bf x}(t))\cdot ({x}_2'(t),-{x_1'(t)})}{|{\bf x}(s)-{\bf x}(t)|^2}\frac{ J_1(k|{\bf x}(s)-{\bf x}(t)|)}{|{\bf x}(s)-{\bf x}(t)|}\:
\frac{|{\bf x}(s)-{\bf x}(t)|^2}{\sin^2\tfrac{s-t}2}\\
B(s,t)&=&\frac{{\rm i}}{4}H_0^1(k|{\bf x}(s)-{\bf x}(t)|)-A(s,t)\log\sin^2\tfrac{s-t}2 \\
D(s,t)&=&\frac{{\rm i}k}{4}H_1^1(k|{\bf x}(s)-{\bf x}(t)|)\frac{ ({\bf x}(s)-{\bf x}(t))\cdot ({x}_2'(t),-{x_1'(t)})}{|{\bf x}(s)-{\bf x}(t)|}
\\
&& 
-C(s,t)\sin^2\tfrac{s-t}2\log\sin^2\tfrac{s-t}2. 
\end{eqnarray*} 
(Observe that ${\rm K}_k$ and ${\rm K}_k^\top$ are transpose to each other). {Very well known properties of the Bessel functions imply that the functions $A,B,C,D$ are smooth functions if so is the map ${\bf x}$, as we have already assumed above.  }
% 
% 
% Here $A,B,C,D$ are {complex} bivariate smooth $2\pi-$periodic functions with $A(s,s)\equiv -\frac1{4\pi}$. We refer the reader to \cite{CVY,KressLI} for detailed expressions of these functions in terms of Hankel functions and the parameterization ${\bf x}$. 

Regarding the parameterized version of the hypersingular operator, the integration-by-parts like formula due to Maue \cite{Maue} (see also \cite{Nedelec}) allows to write ${\rm H}_k$ as the integro-differential operator
\begin{equation}\label{eq:Hk}
 \left[ {\rm H}_k  g \right] (s) =  \left[{\rm D} {\sf V}_k {\rm D}\,g\right](s)
-{\rm i}k^2 \left[{\sf V}_k(({\bf x}'(s)\cdot{\bf x}'(\cdot))g \right](s).
\end{equation}
Here, ${\rm D}\varphi:=\varphi'$ is simply the differentiation operator. 

\subsection{Nystr\"om Discretization}

%In short, apart from the derivative operator, we have mainly to tackle the evaluation of integrals as

The structure of the kernels introduced in the previous section leads to tackle, apart from the derivative operator,  the evaluation of integrals as
\begin{eqnarray}\label{eq:Psi}
\left[\Psi \varphi\right](s):=\int_0^{2\pi} \psi(s-t)a(s,t)\varphi(t)\,{\rm d}t,
\end{eqnarray}
where $a$, $\psi$ are  $2\pi-$periodic, with $a$ being smooth and $\psi$, in principle, singular at $0$.
The operators defined in equation~\eqref{eq:Psi} are $2\pi-$periodic pseudodifferential operators (cf. \cite[Ch.7]{saranen_vainiko}).

\subsubsection{Trigonometric interpolation}

Let us denote 
\[
 \mathbb{T}_N:={\rm span}\:\langle e_n\: : \: -N<n\le N\rangle,\quad \text{with}\quad
 e_n(t):=\exp({\rm i}n t),\quad (n\in\mathbb{Z})
\]
%  for any $n\in\mathbb{Z}$ and set
% \[
%  \mathbb{T}_N:={\rm span}\:\langle \exp({\rm i}n t)\: : \: -N<n\le N\rangle
%  \]
the space of trigonometric polynomials of degree $N$. On $\mathbb{T}_N$ we consider the trigonometric interpolation problem  on the uniform grid $\{j\pi/N\}$:
\[
 \mathbb{T}_N \ni P_Ng\quad  \text{s.t.}\quad  
 (P_Ng)(\tfrac{j\pi}N)=g(\tfrac{j\pi}N),\ j=0,\ldots, 2N-1,
\]
%the  associated  trigonometric interpolation operator at the uniform grid $\{\tfrac{j \pi}{N}\}$. 
The solution of the interpolating problem is given by
\begin{equation}\label{eq:FFT}
 \sum_{n=-N+1}^N \left[\frac{1}{2N}\sum_{m=-N+1}^N g(\tfrac{j\pi}N)e_n(-\tfrac{{\rm i}m\pi}N)\right]e_n({\rm i}n t)
\end{equation}
which can be computed  in ${\cal O}(n\log n)$ operations using FFT.

\subsubsection{Discrete operators}\label{sec:3}

We now introduce 
\begin{equation}
\left[\Psi_{N}\varphi\right](s):=\int_0^{2\pi} \psi(s-t)P_N\left[a(s ,\cdot)\varphi\right](t)\,{\rm d}t\approx \left[\Psi\varphi\right](s)\label{eq:PsiN}
\end{equation}
as discrete approximations of \eqref{eq:Psi}. Clearly $\Psi_N\varphi$ depends only on the pointwise values of the density at the grid points, which justifies the use of the term ``discrete''  when referring to these operators.

Obviously, we are just working with a product-integration rule and the applicability of such procedure relies on being able to compute
\[
 \widehat{\psi}(n):=\frac{1}{2\pi}\int_{0}^{2\pi}\psi(t)e_{-n}(t)\,{\rm d}t,\quad  n\in\mathbb{Z}
\]
i.e. the Fourier coefficients of the weight function $\psi$. Fortunately, for the weight functions featured above, these Fourier coefficients can be computed explicitly. Indeed, for $\psi_1:=\log\sin^2 \tfrac t2$ we have
\begin{eqnarray*}
 \widehat{\psi}_1(n)&=&
 \tfrac1{2\pi}\int_0^{2\pi}\!\log(\sin^2 \tfrac{t}2)\:e_{-n}(t) {\rm d}t=
 \tfrac1{2\pi}\int_0^{2\pi}\!\log(\sin^2 \tfrac{t}2)\:\cos( n t) {\rm d}t\\
 &=&  \begin{cases}
                           -2\log 4,& n=0,\\
                           -2|n|^{-1},&\text{otherwise},
                          \end{cases}
                          \end{eqnarray*}
whereas for $\psi_2:=\sin^2 \tfrac t2\log\sin^2 \tfrac t2$ straightforward calculations yield
                          \begin{eqnarray*}
  \widehat{\psi}_2(n)&=&\tfrac1{2\pi}\int_0^{2\pi}\!\sin^2 \tfrac{t}2\log(\sin^2 \tfrac{t}2)\:e_{-n}(t) {\rm d}t\\
  &=&\tfrac1{8\pi}\int_0^{2\pi}  \log(\sin^2 \tfrac{t}2)\:\left(2\cos(n t)-\cos(n-1)t-\cos(n+1)t\right) {\rm d}t\\
 &=& \begin{cases}
                           \tfrac12,& n=0,\\% [1.1ex]
                           -\tfrac38,& |n|=1,\\% [1.1ex]  
                           \tfrac14\big[\tfrac1{|n+1|}+\tfrac{1}{|n-1|}-\tfrac{2}{|n|}\big],&\text{otherwise}.
                          \end{cases}\\                         
\end{eqnarray*}
We stress that the calculation in the case of the weight $\psi_1$ can be traced back to \cite{kusmaul,martensen} (see also \cite{KressLI}). 
For the remaining case, $\psi_0\equiv 1$, the same approach gives us (see \eqref{eq:FFT})
\begin{eqnarray*}
 \int_0^{2\pi} (P_N g)(t)\,{\rm d}t&=&
 \sum_{n=-N+1}^N \left[\frac{1}{2N}\sum_{m=-N+1}^N g(\tfrac{j\pi}N)\exp(-\tfrac{{\rm i}mn\pi}N)\right]\int_0^{2\pi}\exp(in t)\,{\rm d}t\\
 &=&\frac{\pi}{N}\sum_{j=0}^{2N-1}
 g\big(\tfrac{j\pi}{N}\big),
\end{eqnarray*}
i.e., the trapezoidal rule. Therefore, for $\psi\equiv 1$, we simply have
\[
 \left[\Psi_{N}\varphi\right](s)=\frac{\pi}{N}\sum_{j=0}^{2N-1}
 a\big(s,\tfrac{j\pi}{N}\big)\varphi(\tfrac{j\pi}{N}\big).
\]
% 
% 
% 
% 
% The case $\psi_0\equiv 1$ is even simpler since this approach returns the rectangular rule 
% \[
%  \left[\Psi_{N}\varphi\right](s)=\tfrac{\pi}{N}\sum_{j=0}^{2N-1}
%  a\big(s,\tfrac{j\pi}{N}\big)\varphi(\tfrac{j\pi}{N}\big).
% \]

\subsubsection{Discrete Helmholtz Boundary Integral Operators}

For the single layer operator we work with two types of discretizations. The first one, proposed originally by Kress (cf. \cite{KressLI} and references therein) is simply
\begin{eqnarray}
{[}{\rm V}_{k,N} \varphi{]}(s)&:=&\int_0^{2\pi}  \psi_1(s-t) \left[P_N A(s,\cdot )\varphi\right](t)\,{\rm d}t +\int_0^{2\pi} [P_N B(s,\cdot)\varphi](t){\rm d}t.
\label{eq:VkN}
\end{eqnarray}
One can use the same approach   for the double layer  operator  and obtain
\begin{equation}
{[}{\rm K}_{k,N} \varphi{]}(s):=\!\int_0^{2\pi}\!\!  \psi_1(s-t)[ P_N C(s,\cdot )\sin^2\tfrac{s-\cdot} 2\:\varphi](t) {\rm d}t+\!\int_0^{2\pi} \left[P_N D(s,\cdot)\varphi\right](t){\rm d}t. \label{eq:KN}
\end{equation}
The operator ${\rm K}_{k,N}^\top$ can be defined accordingly. 

Alternatively, we can proceed in a different way and define the {\em more accurate} approximation
\begin{eqnarray}
\label{eq:KNtilde}
{[}\widetilde{\rm K}_{k,N} g{]}(s)&:=&\int_0^{2\pi}\!\! \psi_2(s-t) P_N[C(s,\cdot )g](t)\,{\rm d}t+\int_0^{2\pi} [P_N D(s,\cdot)g ](t){\rm d}t.
\end{eqnarray}
The operator $\widetilde{{\rm K}}^\top_{k,N}$ can be obviously  defined in the same manner. 

We can actually use the same approach for the single layer operator ${\rm V}_k$. Indeed, let us write first
\begin{eqnarray}
 A(s,t)&=&-\frac{1}{4\pi} +\frac{1-
J_0(k |{\bf x}(s)-{\bf x}(t)|)}{4\pi\sin^2\tfrac{s-t}{2}}
%\big[1- J_0(k |{\bf x}(s)-{\bf x}(t)|)\big] 
\sin^2\tfrac{s-t}{2}
%\nonumber\\
% &=:&
=:
 -\frac1{4\pi}+\widetilde{A}(s,t)\sin^2\tfrac{s-t}{2}. \label{eq:Ast}
\end{eqnarray}
We point out that function $A(s,t)$ is smooth with
\[
 \widetilde{A}(s,s)\equiv \frac{k^2}{4\pi} |{\bf x}'(s)|.
\]
Hence, using the Bessel operator defined as
\[
 \left[\Lambda\varphi\right](s):=
 -\frac{1}{4\pi}\int_0^{2\pi }\log\sin^2\tfrac{s-t}{2}\: \varphi(t)\,{\rm d}t,
\]
we have derived the following alternative expression for the single layer operator
%The following expression for the single layer operator {can be now straightforwardly derived}
\begin{eqnarray}\label{eq:VktildeN1}
 {\rm V}_{k}\varphi&=&\Lambda\varphi +   \int_0^{2\pi} \widetilde{A}(\:\cdot\:,t)\psi_2(\:\cdot\:-t)\varphi(t)\,{\rm d}t\nonumber +\int_0^{2\pi} B(\:\cdot\:,t)\varphi(t)\,{\rm d}t =:\Lambda\varphi + {\rm R}_k\varphi,
\end{eqnarray}
which can be exploited to lead to the following approximation 
\begin{eqnarray}
 \big[\widetilde{\rm V}_{k,N} \varphi\big](s) &:=&  
 \big[\Lambda\varphi\big](s) +   \int_0^{2\pi}  \psi_2(s-t)\:P_N[\widetilde{A}(s,\cdot)\varphi](t)\,{\rm d}t\nonumber +\int_0^{2\pi} P_N[\widetilde{B}(s,\cdot)\varphi](t)\,{\rm d}t \\
  &=:&\big[\Lambda\varphi\big](s) + \big[\widetilde{\rm R}_{k,N}\varphi\big](s).\label{eq:VktildeN2}
\end{eqnarray}
Obviously, $\widetilde{\rm V}_{k,N}$ can be applied, in principle, only to trigonometric polynomials, since otherwise the first term gives rise to an infinite series. {As we will see later, this is not a severe constraint for the numerical approximations we propose}. 

Finally, applying integration by parts and making use of the same quadrature rules, we have
\begin{eqnarray}\label{eq:HkN1}
  {\rm H}_{k}  &=& {\rm D}\Lambda {\rm D}   +
 {\rm T}_k 
\end{eqnarray}
with
\[
\big[{\rm T}_k\varphi\big](s) =\int_0^{2\pi} E(s,t)\log\sin^2\tfrac{s-t}{2}\varphi(t)\:{\rm d}t+
\int_0^{2\pi} F(s,t)\varphi(t)\: {\rm d}t
\]
where
\begin{eqnarray*}
 E(s,t)&:=&- \partial_s\partial_t \widetilde{A}(s,t) \sin^2\tfrac{s-t}{2}+
 \tfrac12\big(\partial_s \widetilde{A}(s,t)-\partial_t \widetilde{A}(s,t)\big)\sin(s-t)\\
 &&+\tfrac12\widetilde{A}(s,t)\cos(s-t)-{\rm i}k^2
  \big({\bf x}'(s)\cdot {\bf x}'(t)\big) A(s,t)\\
 F(s,t)&:=&-\partial_s\partial_t {B}(s,t)+\tfrac12 \big(\partial_s \widetilde{A}(s,t)-\partial_t \widetilde{A}(s,t)\big)\sin(s-t)\\
 &&+\widetilde{A}(s,t)(\tfrac12+\cos(s-t))
 -{\rm i}k^2 ( {\bf x}'(s)\cdot {\bf x}'(t)) B(s,t).
 \end{eqnarray*}
 Then, following the same convention, we can define
  \begin{equation}\label{eq:HkN2}
 {\rm H}_{k,N}\varphi :=  {\rm D}\Lambda {\rm D} \varphi  + 
 {\rm T}_{k,N}\varphi
 \end{equation}
 with
 \[
 \big[
 {\rm T}_{k,N}\varphi\big](s)=
\int_0^{2\pi}  \psi_1(s-t) \left[P_N E(s,\cdot )\varphi\right](t){\rm d}t +\int_0^{2\pi} [P_N F(s,\cdot)\varphi](t){\rm d}t\\
\]
Again $ {\rm H}_{k,N}$ is not a full discrete operators, but when applied to trigonometric polynomials it can be computed exactly which turns out to be enough for our purposes. 

%
%where ${\rm T}_k$ is a logarithmic integral operator as ${\rm V}$ (but the kernel is not constant on the diagonal),
% %and  next propose cf. \cite{kressH}  
% \begin{equation}\label{eq:HkN2}
%  {\rm H}_{k,N}    :={\rm D}\Lambda {\rm D}  +
%  {\rm T}_{k,N} . 
% \end{equation}
% where ${\rm T}_{k,N}$ is resembled from ${\rm T}_k$ in the same
% manner as ${\rm V}_{k,N}$ is derived from ${\rm V}_{k}$.
% 
% where $A_2$ and $B_2$ are again suitable smooth $2\pi-$ periodic bivariate functions. The first term in $\widetilde{\rm H}_{k,N}$ can be computed exactly for trigonometric polynomials since
% \[
%  \frac{1}{4\pi}\int_0^{2\pi}\!\! \rho'_{1}(t) \exp({\rm i}nt)\,{\rm d}t=\mathop{\rm sign}(n).
% \]
\subsection{Convergence analysis}

We develop our analysis in periodic Sobolev norms. 
For any $p\in\mathbb{R}$ we first define the  Sobolev norm
\[
 \|\varphi\|^2_{p}:=|\widehat{\varphi}(0)|^2+\sum_{n\ne 0} |n|^{2p}|\widehat{\varphi}(n)|^2.
\]
The periodic Sobolev spaces of order $p$, denoted in what follows by $H^p$, can be defined, for instance, as the completion of trigonometric polynomials in this norm.

We are ready to state the main theorem. The proof follows from application of similar ideas to those introduced in \cite[Ch. 12 and 13]{KressLI} (see also \cite{CVY}). Let us point out that henceforth, for given $A:X\to Y$, we denote by
$\|A\|_{X\to Y}$ its operator norm.

\begin{theorem}\label{theo:01}
Let $p>1/2$ and  $q\ge -1$ with $p+q>1/2$. Then, if 
${\rm A}\in\{{\rm K}_k,{\rm K}_k^\top, {\rm V}_k,{\rm H}_k \}$ and ${\rm A}_N$ is the corresponding approximation, i.e.,
${\rm A}_N\in\{{\rm K}_{k,N},{\rm K}_{k,N}^\top, {\rm V}_{k,N},{\rm H}_{k,N}\}$,%it holds
\begin{eqnarray}\label{eq:01:theo:01}
\|{\rm A}-{\rm A}_{N}\|_{H^{p+q}\to H^p} 
&\le& C_{p,q}N^{-q-\min\{p,1\}}.
\end{eqnarray}
On the other hand, for ${\rm A}\in\{{\rm K}_k,{\rm K}_k^\top, {\rm V}_k \}$ and
$\widetilde{\rm A}_N\in\{\widetilde{\rm K}_{k,N},\widetilde{\rm K}_{k,N}^\top, \widetilde{\rm V}_{k,N}\}$  the corresponding
discretization,  we have for  $q\ge -3$ with $p+q>1/2$. 
\begin{eqnarray}\label{eq:02:theo:01}
 \|{\rm A}-\widetilde{\rm A}_{N}\|_{H^{p+q}\to H^p}\!\!
&\le&\!\! C_{p,q} N^{-q-\min\{p,3\}}.
\end{eqnarray}
\end{theorem}
%\begin{proof}
{\sc Proof} 
%Let $\Psi$ as in \eqref{eq:Psi} and 
For any function $\psi$ we denote the convolution operator in the usual manner:
\[
[\psi*\varphi](s):=\int_{0}^{2\pi}\psi(s-t)\varphi(t)\,{\rm d}t={\sum_{n=-\infty}^\infty
\widehat{\psi}(n)\widehat{\varphi}(n)e_n(s)}.
\]
Then  it is straightforward to check that for $\varphi$ smooth enough, see \eqref{eq:Psi}, 
\begin{eqnarray*}
[\Psi\varphi](s)&=&\int_{0}^{2\pi} a(s,t)\psi(s-t)\varphi(t)\,{\rm d}t=
\sum_{n=-\infty}^\infty a_n(s) [\psi*( e_n \varphi)](s) 
\end{eqnarray*}
where $e_n(s)=\exp({\rm i}n s)$ and 
\[
 a_n(s):=\frac{1}{2\pi}\int_0^{2\pi} a(s,t)e_{-n}(t){\rm d}t
\]
is the $n$th Fourier coefficient of   $a(s,\cdot)$. 
Since function  $a$ is assumed to be smooth, then for any $P$ it holds that
\[
 \sup_{n\in\mathbb{N}}(1+|n|)^{P}\|a_n\|_{L^\infty(0,2\pi)}<C_P. 
\]
Let us restrict ourselves to the cases $\psi=\psi_m$, for $m=0,1,2$  (see the beginning 
of subsection \ref{sec:3}). Denote then by $\Psi_m$ the corresponding operator and by $\Psi_{m,N}$, its  numerical approximation cf. \eqref{eq:PsiN}. 
Clearly, the proof of this Theorem can be reduced to studying 
%
%Then for any ${\rm A}, {\rm A_N}$ as those considered in this theorem, we have
\[
 (\Psi_{m}-\Psi_{m,N})\varphi=\sum_{n=-\infty}^\infty a_n \psi_m*(e_n\varphi-P_N( e_n\varphi)). 
\]
We will make use of the following results:
\begin{itemize}
 \item[(a) ] 
%direct estimates for the Fourier coefficients of $\psi_m$ 
for $m=1,2$ it holds
\[
 \|\psi_m*\varphi\|_{p}\le C_q\|\varphi\|_{p-q},\quad q\le  2m-1                                                       
\] 
whereas for $m=0$
\[
 \|\psi_0*\varphi\|_{p}= 2\pi |\widehat{\varphi}(0)|\le  2\pi \|\varphi\|_{p+q},\quad \forall q\in\mathbb{R};
\]
Indeed, for $m=1,2$
\[
| \widehat{\psi}_m(n)|\le C_m (1+|n|)^{1-2m}
\]
with $C_m$ independent of $m$ which implies
\begin{eqnarray*}
\|\psi_m*\varphi\|_p^2&=&|\widehat{\psi}_m(0)\widehat{\varphi}(0)|^2+\sum_{n=-\infty}^\infty |n|^{2p}|\widehat{\psi}(n)\widehat{\varphi}(n)|^2\\
&\le& C^2_m\left[|
\widehat{\psi}_m(0)\widehat{\varphi}(0)|^2+\sum_{n=-\infty}^\infty |n|^{2(p+1-2m)} |\widehat{\varphi}(n)|^2\right]=C^2_m\|\varphi\|^2_{p-2m+1};
\end{eqnarray*}

\item[(b) ] the convergence estimate for the trigonometric interpolant
 \cite[Theorem 8.2.1]{saranen_vainiko} 
\begin{equation}\label{eq:interp_error}
\|P_N\varphi-\varphi\|_{p}\le C_{p,q} N^{-q}\|\varphi\|_{p+q},\quad 
\forall p,q\ge 0,\quad p+q> 1/2;
\end{equation}
\item[(c) ] the fact that $H^p$ for $p> 1/2$ is an algebra, cf \cite[Lemma 5.13.1]{saranen_vainiko} and therefore
\[
 \|a\varphi\|_p\le C_p\|a\|_p\|\varphi\|_p;
\]
\item[(d) ] the obvious bound $\|e_n\|_p\le \max\{1, |n|^p$\}.
\end{itemize}
We are ready to analyze the approximation error of the discrete operators.  First, for $m=0$, that is, for integral operators with smooth kernel, we have
\begin{eqnarray*}
 \|(\Psi_{0}-\Psi_{0,N})\varphi\|_p &\le&  C \!\sum_{n=-\infty}^\infty\! \|a_n\|_p  
 \|\psi_0*(e_n\varphi-P_N(e_n \varphi))\|_{p}\\
&=& 2\pi C_p\sum_{n=-\infty}^\infty \|a_n\|_p \|e_n\varphi-P_N(e_n \varphi)\|_0\\
&\le& C_{p,q}N^{-p-q}\sum_{n=-\infty}^\infty \|a_n\|_p\|e_n\|_{p+q} \|\varphi\|_{p+q}\\
&\le& C_{p,q}N^{-p-q}\bigg[\sum_{n=-\infty}^\infty \|a_n\|_p (1+|n|)^{p+q}\bigg] \|\varphi\|_{p+q}\\
&\le& C_{p,q}' N^{-p-q}\|\varphi\|_{p+q}
\end{eqnarray*}
for all $p+q\ge 1$. 

Let us examine the case $m=2$. If $p\ge 3$, we can proceed similarly to conclude
\begin{eqnarray*}
 \|(\Psi_{2}-\Psi_{ 2,N})\varphi\|_p&\le& \sum_{n=-\infty}^\infty \|a_n\|_p 
 \|e_n\varphi-P_N(e_n\varphi)\|_{p-3}\\
   &\le&    C_{p,q}N^{ -q-3}\sum_{n=-\infty}^\infty \|a_n\|_p\|e_n\|_{{p+q}} \|\varphi\|_{p+q}
 % &\le&   C_tN^{-t}\sum_{n=-\infty}^\infty \|a_n\|_p (1+|n|)^t \|\varphi\|_t\\
  \le   C_{p,q} N^{ -q-3}\|\varphi\|_{p+q},
\end{eqnarray*}
provided that {$p+q> 1/2$} and $q\ge -3$. %$p+q\ge 1$ 
If $p\in[0,3]$, we can only get convergence estimates for the interpolator in $H^0$ ({we can  not expect} faster convergence in weaker norms). Therefore we have instead
\begin{eqnarray*}
 \|(\Psi_{2}-\Psi_{2,N})\varphi\|_p&\le&  \sum_{n=-\infty}^\infty \|a_n\|_p 
 \|e_n\varphi-P_N(e_n\varphi)\|_{0}\le C_{p,q} N^{ -p-q}\|\varphi\|_{p+q}.
\end{eqnarray*}
Collecting these bounds, the result for $m=3$ follows readily.

Case $m=1$  is left as exercise for the reader. \hfill $\Box$

We recall the functional properties of the boundary operators in the Sobolev setting. Define
\[
 {\cal D}_k:=\begin{bmatrix}
             -{\rm K}_k&{\rm V}_k\\
              -{\rm H}_k&{\rm K}^\top_k              
             \end{bmatrix}.
% :            H^{p+1}\times H^{p}\to 
%              H^{p+1}\times H^{p}
\]
Then, ${\cal D}_k:H^{p+1}\times H^{p}\to  H^{p+1}\times H^{p}$ is continuous for any $p\in\mathbb{R}$. Actually  it holds 
\begin{equation}
 \label{eq:extra:reg}
{\rm K}_k,{\rm K}_k^\top,{\rm R}_k:H^{p}\to H^{p+3}.
\end{equation}
This extra regularizing property has been repeatedly used in the design and analysis  of boundary integral methods for Helmholtz equation. 

If we define 
\begin{eqnarray*}
 {\cal D}_{k,N} := \begin{bmatrix}
              -{\rm K}_{k,N}&{\rm V}_{k,N}\\[1.1ex]
              -{\rm H}_{k,N}&{\rm K}^\top_{k,N}              
             \end{bmatrix}, \quad
 \widetilde{\cal D}_{k,N} :=  \begin{bmatrix}
              -\widetilde{\rm K}_{k,N}&{\rm \widetilde{V}}_{k,N}\\[1.1ex]
              -{\rm H}_{k,N}&\widetilde{\rm K}^\top_{k,N}              
             \end{bmatrix},\quad
\end{eqnarray*}
% \begin{eqnarray*}
%  {\cal D}_{k,N}&:=&\begin{bmatrix} P_N\\
%  &P_N\end{bmatrix}\begin{bmatrix}
%               {\rm K}_{k,N}&{\rm V}_{k,N}\\
%               {\rm H}_{k,N}&{\rm K}^\top_{k,N}              
%              \end{bmatrix},\\
%  \widetilde{\cal D}_{k,N}&:=&\begin{bmatrix} P_N\\
%  &P_N\end{bmatrix}\begin{bmatrix}
%               \widetilde{\rm K}_{k,N}&{\rm \widetilde{V}}_{k,N}\\
%               {\rm H}_{k,N}&\widetilde{\rm K}^\top_{k,N}              
%              \end{bmatrix},\quad
% \end{eqnarray*}
the  following result can be easily derived from Theorem \ref{theo:01}
\begin{proposition} \label{prop:main}
For any $p>1/2$, 
\begin{equation}
\label{eq:01:prop:main} 
{\cal D}_{k,N},  \widetilde{\cal D}_{k,N} : H^{p+1}\times H^{p}\to
H^{p+1}\times H^{p}
\end{equation}
are uniformly continuous. Moreover, if $p> 1/2$ and $q\ge -1$ with $p+q>1/2$, 
\begin{equation}
 \|{\cal D}_{k,N}-{\cal D}_{k}\|_{H^{p+q}\times H^{p+q}\to
H^{p}\times H^{p}}+
\|\widetilde {\cal D}_{k,N}-{\cal D}_{k}\|_{H^{p+q}\times H^{p+q}\to
H^{p}\times H^{p}} \le C N^{-q-\min\{1,p\}},
\label{eq:02:prop:main} 
\end{equation}
and, for $q\ge -2$, $p>1/2$ and $p+q>1/2$, 
\begin{eqnarray}\label{eq:03:prop:main} 
 \|\widetilde{\cal D}_{k,N}-{\cal D}_{k}\|_{H^{p+q+1}\times H^{p+q}\to
H^{p+1}\times H^{p}} \le C {N^{-q-\min\{2,p\}}}.
\end{eqnarray}
\end{proposition}
\begin{proof}
Define 
\begin{equation}\label{eq:defE}
 {\cal E}_{k}:=
\begin{bmatrix}
-{\rm K}_{k}&{\rm V}_{k}\\
-{\rm T}_{k}&{\rm K}^\top _{k}
\end{bmatrix},\quad
 {\cal E}_{k,N}:=
\begin{bmatrix}          
-{\rm K}_{k,N}& {\rm V}_{k,N}\\
-{\rm T}_{k,N}& {\rm K}^\top_{k,N}
\end{bmatrix}.
\end{equation}
Then
\begin{equation}\label{eq:02bb:prop:main} 
 {\cal D}_{k,N}-{\cal D}_k= {\cal E}_{k,N}- {\cal E}_{k}.
\end{equation}
Equation  \eqref{eq:01:theo:01} in Theorem   \ref{theo:01} proves \eqref{eq:02:prop:main} since
\begin{subequations}\label{eq:est:for:Ek}
\begin{eqnarray}
 \|{\rm K}_{k,N}-{\rm K}_{k}\|_{H^{p+q}\to \times H^{p}}&\le& C N^{-q-\min\{1,p\}}\\
 \|{\rm K}^\top_{k,N}-{\rm K}^\top_{k}\|_{H^{p+q}\to H^{p}}&\le& C N^{-q-\min\{1,p\}}\\
 \| {\rm V} _{k,N}-{\rm V} _{k}\|_{H^{p+q}\to H^{p}}&\le& C N^{-q-1-\min\{1,p\}}\\
% \|{\rm H}_{k,N}-{\rm H}_{k}\|_{H^{p+q}\to H^{p}}&=&
 \|{\rm T}_{k,N}-{\rm T}_{k}\|_{H^{p+q}\to H^{p}}&\le& 
 C N^{-q-\min\{1,p\}}\label{eq:d:est:for:Ek}
\end{eqnarray}
\end{subequations}
which hold for $p+q>1/2$ and $q\ge -1$. Moreover, from the mapping properties of the continuous operators, these estimates with $q=0$ imply the first result for ${\cal D}_{k,N}$. 

For the second estimate, we start now from
\begin{equation}\label{eq:03a:prop:main} 
 \widetilde{\cal D}_{k,N}-{\cal D}_k=  \widetilde{\cal F}_{k,N}- {\cal F}_{k}
\end{equation}
where
\begin{equation}\label{eq:defF}
 {\cal F}_{k}:=
\begin{bmatrix}
-{\rm K}_{k}&{\rm R}_{k}\\
-{\rm T}_{k}&{\rm K}^\top_{k}
\end{bmatrix},\quad
 \widetilde{\cal F}_{k,N}:=
\begin{bmatrix}
-\widetilde{\rm K}_{k,N}& \widetilde{\rm R}_{k,N}\\
-{\rm T}_{k,N}&\widetilde{\rm K}^\top_{k,N}
\end{bmatrix},
\end{equation}
for which we have the error convergence estimates 
\begin{subequations}\label{eq:est:for:Fk}
 \begin{eqnarray}
 \|\widetilde{\rm K}_{k,N}-{\rm K}_{k}\|_{H^{p'+q'}\to H^{p'}}&\le& C N^{-q'-\min\{3,p'\}}
 ,\quad q'\ge -3 \label{eq:a:est:for:Fk} \\
 \|\widetilde{\rm K}^\top_{k,N}-{\rm K}^\top_{k}\|_{H^{p'+q'}\to  H^{p'}}&\le& C N^{-q'-\min\{3,p'\}},\quad q'\ge -3\label{eq:b:est:for:Fk}\\
 \|\widetilde{\rm R}_{k,N}-{\rm R}_{k}\|_{H^{p'+q'}\to H^{p'}}&\le& 
 C N^{-q'-\min\{3,p'\}},\quad q'\ge -3 \label{eq:c:est:for:Fk}\\
% \|{\rm H}_{k,N}-{\rm H}_{k}\|_{H^{p+q+1}\to H^{p} }&&\\
% &&\hspace{-2cm}\:=\:
  \|{\rm T}_{k,N}-{\rm T}_{k}\|_{H^{p'+q'}\times H^{p'} }&\le& 
 C N^{-q'-\min\{1,p'\}},\quad q'\ge -1.\label{eq:d:est:for:Fk}
\end{eqnarray}
\end{subequations}
(With the restriction $p'+q'>1/2$ in all these cases). 
Choosing   $q'=q$ and $p'=p$ in all the estimates in \eqref{eq:est:for:Fk} we get    \eqref{eq:02:prop:main} which, in particular, implies \eqref{eq:01:prop:main} as a simple consequence. To prove \eqref{eq:03:prop:main}, we take $(p',q')=(p+1,q)$ in \eqref{eq:a:est:for:Fk},
$(p',q')=(p,q)$ in \eqref{eq:b:est:for:Fk}, $(p',q')=(p+1,q-1)$ in 
\eqref{eq:c:est:for:Fk} and $(p',q')=(p,q+1)$ in \eqref{eq:d:est:for:Fk}. 
% where $p>1/2$ and $p+q>1/2$.
% and
% \begin{eqnarray*}
%  \|\widetilde{\rm K}_{k,N}-{\rm K}_{k}\|_{H^{p+q+1}\to H^{p+1}}&\le& C N^{-q-\min\{3,p\}}\\
%  \|\widetilde{\rm K}^\top_{k,N}-{\rm K}^\top_{k}\|_{H^{p+q}\times H^{p}\to H^{p+q}\times H^{p}}&\le& C N^{-q-\min\{3,p\}}\\
%  \|{\rm R}_{k,N}-{\rm R}_{k}\|_{H^{p+q}\to H^{p+1}}&\le& 
%  C N^{-q-\min\{2,p\}}\\
% % \|{\rm H}_{k,N}-{\rm H}_{k}\|_{H^{p+q+1}\to H^{p} }&&\\
% % &&\hspace{-2cm}\:=\:
%   \|{\rm T}_{k,N}-{\rm T}_{k}\|_{H^{p+q+1}\times H^{p} }&\le& 
%  C N^{-q-1-\min\{1,p\}}
% \end{eqnarray*}
\end{proof}

In short, we have shown in this section two 
different types of discrete versions of the Helmholtz boundary layer operators.  The first type of discretization is simpler and works well for equations stated in $H^p\times H^p$ such as the equations of the second kind where the hypersingular operator is not the leading term, either because it does not appear or because the strong singular part is canceled out.  The second type of discretization involving the operators $ \widetilde{\cal D}_{k,N}$ turns out to be more appropriate for formulations in the natural space $H^{p+1}\times H^{p}$ or for complex formulations where  the operators are more involved and/or the operator ${\rm H}_k$ plays a dominant role. {Actually, we could keep ${\rm K}_{k,N}$ and ${\rm K}^\top_{k,N}$ in $\widetilde{\cal D}_{k,N}$ and  the  desired convergence property, namely  $\|\widetilde{\cal D}_{k,N}-{\cal D}_{k}\|_{H^{p}\times H^{p+1}\to
H^{p}\times H^{p+1}}\to 0$ for any $p>1/2$, still holds. We have prefered, however, to collect in $\widetilde{\cal D}_{k,N}$ the more accurate discretization. }
We will consider several examples of these cases in next section. 

\section{Boundary integral equations for transmission problems and their Nystr\"om discretizations}

We consider numerical approximations of several well-posed formulations of the transmission problem \eqref{eq:Tr} presented in Section 2. 
Equipped with the discrete operators introduced and analyzed  in the previous section, the stability and convergence of the resulting schemes can be now easily proven.

For the sake of a simpler notation, we will denote in this section only by ${\rm V}_{\pm},$
${\rm H}_{\pm},$ etc the corresponding layer operators for $k_\pm$. Their discrete versions will be denoted, as before, by simply adding the subscript $N$.

First we consider the {Kress-Roach}  formulation cf \cite{kressRoach}. Defining
\[
{\cal L}_1\!\begin{bmatrix} a\\
           \varphi
          \end{bmatrix}:=
\left(\!
 \tfrac{1+\nu}{2}{\cal I}
\!+\!\begin{bmatrix}
  \nu {\rm K}_{-}-{\rm K}_{+}& {\rm V}_{+}-{\rm V}_{-}\\
  \nu ({\rm H}_{-}-{\rm H}_{+})& \nu {\rm K}^\top_{+}-{\rm K}^\top_{-}\\
 \end{bmatrix}\!\right)\begin{bmatrix} a\\
           \varphi
          \end{bmatrix},
\]
where ${\cal I}$ is the identity operator matrix, this formulation amounts to solving the system of boundary equations 
\begin{equation}\label{eq:L1}
 {\cal L}_1\begin{bmatrix}
            a\\
            \varphi
           \end{bmatrix}=\begin{bmatrix}
           f\\
           \lambda\end{bmatrix}.
           %            \begin{bmatrix}
%            -(\gamma^+{u}_{\rm inc})\circ{\bf x}\\
%            - \nu|{\bf x}'|(\partial_n^+{u}_{\rm inc})\circ{\bf x}
%            \end{bmatrix}
\end{equation}
It is well known that if $(f,\lambda)=(h,\nu\,\eta)$ cf. \eqref{eq:rhs}, then the unique  solution is $a= u_{\rm t}\circ{\bf x}$, 
$\varphi=|{\bf x}'|(\partial_n u_{\rm t})\circ{\bf x}$ where $u_t$ is exterior part of the total wave:
$u_{\rm t}=u_++u_{\rm inc}$.  Clearly, once this equation is solved, taking into account  the transmission conditions \eqref{eq:Tr}, we can evaluate $u^\pm$ by means of 
\eqref{eq:pot}.

The discrete versions of the operators $\mathcal{L}_{1}$ are given by
\begin{eqnarray*}
 {\cal L}_{1,N}\!\!&:=&\!
 \tfrac{1+\nu}{2}{\cal I}
+{\cal P}_N
\begin{bmatrix}
  \nu {\rm K}_{-,N}-{\rm K}_{+,N}& {\rm V}_{+,N}-{\rm V}_{-,N}\\
  \nu ({\rm H}_{-,N}-{\rm H}_{+,N})& \nu {\rm K}^\top_{+,N}-{\rm K}^\top_{-,N}
 \end{bmatrix}\! \\
 &=&\! \tfrac{1+\nu}{2}{\cal I}
+{\cal P}_N
\begin{bmatrix}
  \nu {\rm K}_{-,N}-{\rm K}_{+,N}& {\rm V}_{+,N}-{\rm V}_{-,N}\\
  \nu ({\rm T}_{-,N}-{\rm T}_{+,N})& \nu {\rm K}^\top_{+,N}-{\rm K}^\top_{-,N}
 \end{bmatrix}\! \\
\end{eqnarray*}
(recall \eqref{eq:HkN1}-\eqref{eq:HkN2})
where
\[
 {\cal P}_N=\begin{bmatrix}
             P_N\\
             &P_N
            \end{bmatrix}.
\]
Thus, the discrete problem is given by
\begin{equation}\label{eq:L1N} 
 {\cal L}_{1,N}\begin{bmatrix}
            a_N\\
            \varphi_N
           \end{bmatrix}=\begin{bmatrix}
           P_N f \\
           P_N \lambda 
           \end{bmatrix}
           \end{equation}
Observe that the last equation implies that $(a_N,\varphi_N)\in\mathbb{T}_N\times\mathbb{T}_N$ which allows us to reformulate the method as a true Nystr\"om scheme, where the unknowns are the pointwise values of the densities at the grid points $\{\tfrac{j\pi}{N}\}$. 

We will consider next  the Costabel-Stephan formulation \cite{CoSt}: Let
\begin{eqnarray*}
 {\cal L}_{2}&:=&
\begin{bmatrix}
 -({\rm K}_{-}+{\rm K}_{+})& \nu^{-1}{\rm V}_{+}+{\rm V}_{-}\\
- ({\rm H}_{-}+\nu {\rm H}_{+})&  {\rm K}^\top_{+}+{\rm K}^\top_{-}
 \end{bmatrix}\\
 &=&
            \!(1+\nu^{-1})\begin{bmatrix}
                 &\Lambda\\
                 -\nu \mathrm{D\Lambda D}                 
                \end{bmatrix}+
\begin{bmatrix}
   -  {\rm K}_{-}-  {\rm K}_{+}& {\nu}^{-1}  \widetilde{\rm R}_{+}+ \widetilde{\rm R}_{-}\\
 -({\rm T}_{-}+\nu {\rm T}_{+})&   {\rm K}^\top_{+}+ {\rm K}^\top_{-}
 \end{bmatrix}
 \end{eqnarray*}
 and the associated system of integral equations 
\begin{equation}\label{eq:L2}
 {\cal L}_2\begin{bmatrix}
 a\\
 \varphi
           \end{bmatrix}  =\begin{bmatrix}
           f \\
            \lambda\end{bmatrix}.
            %,\quad \text{with } u_{\rm t}:=u_++u_{\rm inc}.
\end{equation}
In this case, if we take $(f,\lambda)=(h, \eta)$, then
$(a,\varphi)= \left(u_{\rm t}\circ{\bf x},|{\bf x}'|(\partial_n u_{\rm t})\circ{\bf x}\right)$ is again the exact solution.

Letting 
\begin{eqnarray*}
\widetilde{\cal L}_{2,N}\!\!&:=&\!(1+\nu^{-1})\begin{bmatrix}
                 &\Lambda\\
                 -\nu\mathrm{D\Lambda D}                 
                \end{bmatrix}+
{\cal P}_N
\begin{bmatrix}
   - \widetilde{\rm K}_{-,N}- \widetilde{\rm K}_{+,N}& {\nu}^{-1}  \widetilde{\rm R}_{+,N}+ \widetilde{\rm R}_{-,N}\\
 -({\rm T}_{-,N}+\nu {\rm T}_{+,N})&  \widetilde{\rm K}^\top_{+,N}+ \widetilde{\rm K}^\top_{-,N}
 \end{bmatrix}, 
\end{eqnarray*}
 the method we propose for solving \eqref{eq:L2} can be written in operational  form as follows
\begin{equation}\label{eq:L2N} 
 \widetilde{\cal L}_{2,N}\begin{bmatrix}
            a_N\\
            \varphi_N
           \end{bmatrix}=\begin{bmatrix}
           P_N f \\
           P_N \lambda 
           \end{bmatrix}. 
\end{equation}
As before, $(a_N,\varphi_N)\in\mathbb{T}_N\times\mathbb{T}_N$ for any pair $(f,\lambda)$ on the right hand side. (This can be easily 
seen by noticing that the leading part in $\widetilde{{\cal L}}_{2,N}$ is diagonal in the complex exponential bases).

The so-called regularized combined field integral equation, proposed in \cite{{CMV}} will be also analyzed here. Let 
\begin{eqnarray*}
{\cal L}_3=\frac{1}{\nu+1} {\cal L}_1 +\frac{2}{\nu+1}\begin{bmatrix}
      &{\rm V}_{\kappa}\\
      -\nu {\rm H}_{\kappa}
      \end{bmatrix}{\cal L}_2=
\begin{bmatrix}
 \frac{1}2I+{\rm K}_{-}&-\nu^{-1}{\rm V}_-\\
\nu {\rm H}_-& \frac{1}2I
-{\rm K}_{-}^\top&
\end{bmatrix}+{\mathcal R}_{\kappa}{\cal L}_2
%\\
%&&+{\mathcal R}_{\kappa}{\cal L}_2
%+{\mathcal R}_{\kappa}\begin{bmatrix}
% -({\rm K}_{-}+{\rm K}_{+})& \nu^{-1}{\rm V}_{+}+{\rm V}_{-}\\
%- ({\rm H}_{-}+\nu {\rm H}_{+})&  {\rm K}^\top_{+}+{\rm K}^\top_{-}
% \end{bmatrix}.
\end{eqnarray*}
with 
\[
 {\cal R}_{\kappa}:= \frac{1}{\nu+1}\begin{bmatrix}
      I&2{\rm V}_{\kappa}\\
      -2\nu {\rm H}_\kappa&\nu I
      \end{bmatrix}.
      \]
% Then
% \begin{eqnarray*}
%  {\cal L}_3&:=&
%  %\frac{\nu}{\nu+1} {\cal L}_1 +\frac{2}{\nu+1}\begin{bmatrix}
%  %     &S_{\kappa}\\
%  %     -\nu N_{\kappa}
%  %     \end{bmatrix} {\cal L}_2 \\
% %           &
% % =&
% \begin{bmatrix}
%  \frac{1}2I+{\rm K}_{-}&-\nu^{-1}{\rm V}_-\\
% \nu {\rm H}_-& \frac{1}2I
% -{\rm K}_{-}^\top&
% \end{bmatrix}+{\mathcal R}_{\kappa}{\cal L}_2
% =\frac{\nu}{\nu+1} {\cal L}_1 +\frac{2}{\nu+1}\begin{bmatrix}
%       &S_{\kappa}\\
%       -\nu N_{\kappa}
%       \end{bmatrix} {\cal L}_2
% %\\
% %&&+{\mathcal R}_{\kappa}{\cal L}_2
% %+{\mathcal R}_{\kappa}\begin{bmatrix}
% % -({\rm K}_{-}+{\rm K}_{+})& \nu^{-1}{\rm V}_{+}+{\rm V}_{-}\\
% %- ({\rm H}_{-}+\nu {\rm H}_{+})&  {\rm K}^\top_{+}+{\rm K}^\top_{-}
% % \end{bmatrix}.
% \end{eqnarray*}
The boundary integral equation  is then given by 
\begin{equation}\label{eq:L3}
 {\cal L}_3\begin{bmatrix}
 a\\
 \varphi
           \end{bmatrix}  =\mathcal{R}_{\kappa}\begin{bmatrix}
           f \\
            \lambda \end{bmatrix},\quad
\end{equation}
It can be shown (see \cite{CMV})
that this system of integral equations admits a unique solution provided that ${\kappa}$ is chosen to be a complex number with positive imaginary part. Moreover, this parameter can be adjusted to make eigenvalues cluster around $1$.  Besides, by construction if we plug 
$(h, \eta)$ in the right hand side, the unique solution is $(u_{\rm t}\circ{\bf x}
,|{\bf x}'|(\partial_n u_{\rm t})\circ{\bf x})$. In other words, this is a new direct method where ${\cal R}_k$ works as some sort of preconditioner for  ${\cal L}_2$. 

The discretization of the regularized equations is done as follows. First, we set
\[
{\mathcal R}_{\kappa,N}:=\frac{1}{\nu+1}\begin{bmatrix}
      I&2   \Lambda +P_N \widetilde{\mathrm{R}}_{\kappa,N}\\
      -2\nu \mathrm{D\Lambda D}-2\nu P_N \mathrm{T}_{\kappa,N} &\nu I
      \end{bmatrix} 
\]
and next we define
\begin{eqnarray*}
 \widetilde{\cal L}_{3,N}\!\!&:=&\!\!\!\!
 %\frac{\nu}{\nu+1} {\cal L}_1 +\frac{2}{\nu+1}\begin{bmatrix}
 %     &S_{\kappa}\\
 %     -\nu N_{\kappa}
 %     \end{bmatrix} {\cal L}_2 \\
%           &
% =&
\begin{bmatrix}
 \frac{1}2I+P_N \widetilde{\rm K}_{-,N}&-\nu^{-1}\Lambda -\nu^{-1} P_N \widetilde{\rm R}_{-,N}\\
\nu \mathrm{D}\Lambda\mathrm{D}+ \nu P_N {\rm T}_{-,N} & \frac{1}2I
-P_N \widetilde{\rm K}_{-,N}^\top
\end{bmatrix}+{\mathcal R}_{\kappa,N}\widetilde{\cal L}_{2,N}\\
&=& 
\begin{bmatrix}
 \frac{1}2I&-\nu^{-1}\Lambda \\
\nu \mathrm{D}\Lambda\mathrm{D} & \frac{1}2I
\end{bmatrix}+{\cal P}_N\begin{bmatrix}
   \widetilde{\rm K}_{-,N}&   \nu^{-1}\widetilde{\rm R}_{-,N}\\
 \nu {\rm T}_{-,N} & - \widetilde{\rm K}_{-,N}^\top
\end{bmatrix}+{\mathcal R}_{\kappa,N}\widetilde{\cal L}_{2,N}
%\!\begin{bmatrix}
% -(\widetilde{\rm K}_{-,N}+\widetilde{\rm K}_{+,N})\!\!& \nu^{-1}{\rm V}_{+,N}+{\rm V}_{-,N}\\
%- ({\rm H}_{-,N}+\nu {\rm H}_{+,N})\!\!&  \widetilde{\rm K}^\top_{+,N}+\widetilde{\rm K}^\top_{-,N}
% \end{bmatrix}.
\end{eqnarray*}
(Observe that the first matrix operator maps $\mathbb{T}_N\times \mathbb{T}_N$ into itself.)
The numerical algorithm, in operator form, is given by
\begin{equation}\label{eq:L3N} 
 \widetilde{\cal L}_{3,N}\begin{bmatrix}
                a_N\\
                \varphi_N
               \end{bmatrix}=
               {\cal R}_{\kappa,N}
               \begin{bmatrix}
                P_N f\\
                P_N \lambda
               \end{bmatrix}.
\end{equation}
Observe again that the right-hand-sides are trigonometric polynomials, and thus so are the solutions of these discrete problems . 
% 
% Alternatively, we can use  $\widetilde{\rm K}_{\pm,N}$, $\widetilde{\rm K}^\top_{\pm,N}$ instead leading to a new discretization  which will be denoted
% by $\widetilde{{\cal L}}_{3,N}$:
% \begin{equation}
% \label{eq:L3Ntilde} \widetilde{\cal L}_{3,N}\begin{bmatrix}
%                 a_N\\
%                 \varphi_N
%                \end{bmatrix}=
%                {\cal R}_{\kappa,N}
%                \begin{bmatrix}
%                 P_N f\\
%                 P_N \lambda
%                \end{bmatrix}.
% \end{equation}

We also investigate an integral formulation based on an indirect method. That is, unlike the formulations considered so far, the unknown is not immediately related to traces on the boundary of the solution of the transmission problem. This integral formulation, has an interesting feature: the solution of the transmission Helmholtz problem can be reconstructed from knowledge of one boundary density only. In other words, this integral equation needs half as many unknowns as the other integral formulations considered in this paper thus far.  Let us describe this equation, which was first introduced
in \cite{KlMa}. We seek a function $\mu$ so that 
\[
u^- =-2[{\rm SL_-}\mu],\quad u^+=\nu {\rm SL}_+(I+2{\rm K}_-^\top)\mu-2{\rm DL_+}{\rm V_-}\mu
\] 
(${\rm SL_\pm}$ and ${\rm DL_\pm}$ are the corresponding parameterized layer potentials). 
The density $\mu$ can be computed by solving the boundary integral {\em equation}  
%\begin{equation}\label{eq:single}
\begin{eqnarray}\label{eq:L4}
L_4 \:\mu&:=&-\frac{\nu+1}{2}\mu+\mathbf{K}\mu-{\rm i}\rho\mathbf{V}\mu= f,%|{\bf x}'|(\partial_n u^{\rm inc})\circ{\bf x}-{\rm i}\rho (\gamma u^{\rm inc})\circ{\bf x }
%\frac{\partial u^{inc}}{\partial n}-i\rho u^{inc}
\end{eqnarray}
where $f=\lambda-{\rm i \rho} g$.
Here $\rho$ is a coupling parameter which must be real and different from zero to ensure
the well-posedness of the equation. In the definition of the operator $L_4$ we used the operators
\[
\mathbf{K}:=-{\rm K}_{-}^\top(\nu I-2{\rm K}_{-}^\top)-\nu {\rm K}_{+}^\top(I+2{\rm K}_{-}^\top)+2({\rm H}_{+}-{\rm H}_{-}){\rm V}_{-}
\]
and
\[
\mathbf{V}:=-\nu {\rm V}_+(I+2{\rm K}_{-}^\top)-(I-2{\rm K}_{+}){\rm V}_{-}.
\]
The discretizations of these operators are given by
\begin{eqnarray*}
 \mathbf{K}_{N}&:=&-{\rm K}_{-,N}^\top(\nu I-2{\rm K}_{-,N}^\top)-\nu 
 {\rm K}_{+,N}^\top(I+2{\rm K}_{-,N}^\top)+2 ({\rm T}_{+,N}-{\rm T}_{-,N}){\rm V}_{-,N}\big]\\
 \mathbf{V}_{N}&:=&-\nu {\rm V}_{+,N}(I+2{\rm K}_{-,N}^\top)-(I-2{\rm K}_{+,N}){\rm V}_{-,N}.
 \end{eqnarray*}
Thus, we define
\[
 L_{4,N} :=-\tfrac{\nu+1}{2}I+P_N\mathbf{K}_N-{\rm i}\rho P_N\mathbf{V}_N.
\]
and the discretization of the equation $L_4\mu=f$ is given by
\begin{equation}\label{eq:L4N}
 L_{4,N} \mu_N=P_N f. 
\end{equation}
Again, $\mu_N\in\mathbb{T}_N$ regardless of the right hand side $f$.

\begin{theorem}\label{theo:main02}
The mappings  \eqref{eq:L1}, \eqref{eq:L2}, \eqref{eq:L3} and \eqref{eq:L4} 
\begin{equation}\label{eq:00:theo:main02}
 \begin{aligned}
  {\cal L}_1:&H^p\times H^p\to H^p\times H^p,\quad j=1,3\\
  {\cal L}_j:&H^{p+1}\times H^p\to H^{p+1}\times H^p,\quad j=1,2,3\\
 % {\cal L}_3:&H^{p+1}\times H^{p}\to H^{p+1}\times H^p,\\ 
%  {\cal L}_3:&H^{p+1}\times H^{p+1}\to 
%  H^{p+1}\times H^{p+1},\\
  L_4:&H^{p}\to H^p
 \end{aligned}
\end{equation}
are continuous and invertible for all $p\in \mathbb{R}$. 

Moreover, for $p> 1/2$, 
\begin{subequations}
 \label{eq:01:theo:main02}
\begin{eqnarray}
 \|{\cal L}_1-{\cal L}_{1,N}\|_{H^p\times H^p\to H^p\times H^p }&\le& C_p N^{-\min\{p,1\}},\label{eq:01a:theo:main02}\\
 \|{\cal L}_2-\widetilde{\cal L}_{2,N}\|_{H^{p+1}\times H^{p}\to H^{p}\times H^{p+1} }&\le& C_p N^{-\min\{p,2\}}, \label{eq:01b:theo:main02}\\
 \|{\cal L}_3-\widetilde{\cal L}_{3,N}\|_{H^{p+1}\times H^{p}\to H^{p}\times H^{p+1} }&\le& C_p N^{-\min\{p,1\}}, \\
 \|{\cal L}_3-\widetilde{\cal L}_{3,N}\|_{H^{p+1}\times H^{p+1}\to H^{p+1}\times H^{p+1} }&\le& C_p N^{-\min\{p,1\}},\label{eq:01d:theo:main02}
 \\
 %\|{\cal L}_2-{\cal L}_{2,N}\|_{H^{s+1}\times H^s\to H^{s+1}\times H^s }&\le& C_s N^{-1}\\
 \|{ L}_4-{ L}_{4,N}\|_{H^{p}\to H^{p}}&\le& C_p N^{-\min\{p,1\}}.
\end{eqnarray}
\end{subequations}
Furthermore, we have the following convergence results: For all $p> 1/2$ and $q\ge 0$, if $(a^1,\varphi^1)$ denotes the 
exact solution for \eqref{eq:L1} and
$(a^1_N,\varphi_1^N)$ is the corresponding numerical solution of {\rm
\eqref{eq:L1N}}, %% \ref{eq:L2N},
it holds  
\begin{subequations} 
 \label{eq:02:theo:main02}
\begin{eqnarray}
 \|a -a^1_N\|_{p}+
 \|\varphi -\varphi_N^1\|_{p} &\le&  C N^{-q}\left[\|a\|_{p+q}+\|\varphi\|_{p+q} \right]\label{eq:02a:theo:main02}
%  \\
%  \|a-\widetilde{a}_N^3\|_{p+1}+
%  \|\varphi -\widetilde{\varphi}_N^3\|_{p} &\le&  C N^{-q}\left[\|f\|_{p+q+1}+\|\lambda\|_{p+q} \right],
 %\\
 %&&\qquad\qquad j=2,3. 
\end{eqnarray}
\end{subequations}
Let for $j=2,3$  $(\widetilde{a}^j_N,\widetilde{\varphi}^j_N)$ the continuous solution of 
 \eqref{eq:L2} and \eqref{eq:L3} and    $(\widetilde{a}^j_N,\widetilde{\varphi}^j_N)$ the discrete solution of \eqref{eq:L2N} and \eqref{eq:L3N}. Then we have
\begin{subequations}
 \label{eq:03:theo:main02}
\begin{eqnarray}
  \|a-\widetilde{a}_N^2\|_{p+1}+
 \|\varphi-\widetilde{\varphi}^2_N\|_{p}&\le&C N^{ -q}\left[\|a\|_{p+q+1}+\|\varphi\|_{p+q} \right]. \label{eq:03a:theo:main02}\\
\|a-\widetilde{a}_N^3\|_{p+1}+
 \|\varphi-\widetilde{\varphi}^3_N\|_{p}&\le & C N^{-q}\left[\|f\|_{p+q+1}+\|\lambda\|_{p+q} \right].\label{eq:03b1:theo:main02}\\
\|a-\widetilde{a}_N^3\|_{p+1}+
 \|\varphi-\widetilde{\varphi}^3_N\|_{p+1}&\le & C N^{-q}\left[\|f\|_{p+q+1}+\|\lambda\|_{p+q+1} \right].\label{eq:03b:theo:main02}
\end{eqnarray}
\end{subequations}
Finally if $\mu$ is the solution of ${ L}_4$ and $\mu_N$ that given by  the numerical scheme  \eqref{eq:L4N}, 
\[
 \|\mu-\mu_N\|_p\le C N^{-q}\|\mu\|_{p+q},\quad  p>1/2,\ q\ge 0.
\]
In the estimates above, $C>0$ is independent of $a,\varphi$, $f$, $\lambda$ or $\mu$,  and $N$. 
\end{theorem}
\begin{proof}
The functional properties stated in \eqref{eq:00:theo:main02} are well known and can be easily derived from the functional properties of the operators involved (see Proposition \ref{prop:main}).

 The proofs for all the convergence estimates share the same ideas. Thus, for the sake of brevity we restrict ourselves to consider a few representative cases to illustrate the kind of techniques used here.
 
 \
 
\noindent {\bf Proof of \eqref{eq:01a:theo:main02} and \eqref{eq:02a:theo:main02}.} 
Denote as in \eqref{eq:defE}
\[
{\cal E}_\pm :=
\begin{bmatrix}
-{\rm K}_\pm&{\rm V}_\pm\\
-{\rm T}_\pm&{\rm K}^\top_\pm
\end{bmatrix} 
%\quad \text{so that}\quad                      
%                     {\cal D}_k=\begin{bmatrix}
%                        0\ &0\ \\
%                       {\rm D}\Lambda{\rm D}\ &0\ 
%                       \end{bmatrix}+{\cal E}_k
  \]
Notice that ${\cal E}_\pm:H^{p}\times H^{p}\to H^{p+1}\times H^{p+1}$ and therefore, from 
from \eqref{eq:interp_error}, 
\begin{equation}\label{eq:00:Ek-Ekn}
 \left\|({\cal P}_N-{\cal I}){\cal E}_\pm\right\|_{H^{p+q}\times H^{p+q}\to H^{p}\times H^{p}}
                      \le C  N^{-q-1}
 \end{equation}
for any $p\ge 0$, $q\ge -1$ with $p+q>1/2$. Setting accordingly
\[
 {\cal E}_{\pm,N}:=
\begin{bmatrix}
-{\rm K}_{\pm,N}&{\rm V}_{\pm,N}\\
-{\rm T}_{\pm,N}&{\rm K}_{\pm,N}
\end{bmatrix} 
\]
we notice that cf \eqref{eq:02:prop:main} (see also \eqref{eq:02bb:prop:main})
\begin{equation}\label{eq:Ek-Ekn}
 \|{\cal E}_k-{\cal E}_{k,N}\|_{H^{p+q}\times H^{p+q}\to H^{p}\times H^p}\le C N^{-q-\min\{1,p\}}.
\end{equation}
On the other hand,
\begin{eqnarray*}
{\cal L}_{1}&=&\tfrac{1+\nu}{2} {\cal I} +
                                               \begin{bmatrix}
                                              1&\\
                                              & \nu
                                             \end{bmatrix}{\cal E}_{+}- {\cal E}_{-}\begin{bmatrix}
                                              \nu&\\
                                              & 1
                                             \end{bmatrix},
                                             \quad \\
{\cal L}_{1,N}&=&       \tfrac{1+\nu}{2} {\cal I} -
                                               {\cal P}_N\begin{bmatrix}
                                              1&\\
                                              & \nu
                                             \end{bmatrix}{\cal E}_{+,N}+ {\cal P}_N{\cal E}_{-,N}\begin{bmatrix}
                                              \nu&\\
                                              & 1
                                             \end{bmatrix}.                                      
\end{eqnarray*}
Therefore, \eqref{eq:00:Ek-Ekn} and \eqref{eq:Ek-Ekn} yield
\begin{eqnarray}
\|{\cal L}_1-{\cal L}_{1,N}\|_{H^{p+q}\times H^{p+q}\to H^p\times H^p }&&
\nonumber \\
&&\hspace{-3.0cm} \le \max\{\nu,1\}
\bigg[
\|({\cal P}_N-{\cal I}) {\cal E}_{\pm } \|_{H^{p+q}\times H^{p+q}\to H^p\times H^p }
\nonumber \\
%+
%\|({\cal P}_N-{\cal I}) {\cal E}_{-} \|_{H^{p+q}\times H^{p+q}\to H^p\times H^p }\\
&&\hspace{-2.7cm} +
\|{\cal P}_N\|_{H^{p}\times H^{p}\to H^p\times H^p }
%\times\\
%&&\hspace{-4.6cm}
\big[\|
 {\cal E}_{\pm }-{\cal E}_{\pm ,N}\|_{H^{p+q}\times H^{p+q}\to H^p\times H^p }
 %+\| {\cal D}_{-}-{\cal D}_{+,N}\|_{H^{p+q}\times H^{p+q}\to H^p\times H^p }
\big]\nonumber \\
&&\hspace{-3.0cm}\le  CN^{-q-\min\{1,p\}}.\label{eq:L1:00}
\end{eqnarray}
In particular, setting $q=0$  implies \eqref{eq:01a:theo:main02}. The error estimate for the numerical method is obtained using standard techniques:
\begin{eqnarray*}
 \left\|\begin{bmatrix}
   a -a^1_N\\
   \varphi -\varphi^1_N
   \end{bmatrix}\right\|_p&\le& C  \left\|{\cal L}_{1,N}\begin{bmatrix}
a -a^1_N\\
   \varphi -\varphi^1_N
   \end{bmatrix}\right\|_p\\
   \\
   &\le& \left\|({\cal L}_{1,N}-{\cal L}_{1})\begin{bmatrix}
a\\
   \varphi
   \end{bmatrix}\right\|_p+
   \left\|{\cal L}_{1}\begin{bmatrix}
a\\
   \varphi
   \end{bmatrix}-{\cal L}_{1,N}\begin{bmatrix}
a_N\\
   \varphi_N
   \end{bmatrix}\right\|_p \\
   &\le& \left\|({\cal L}_{1,N}-{\cal L}_{1})\begin{bmatrix}
a\\
   \varphi
   \end{bmatrix}\right\|_p+
   \left\|({\cal I}-{\cal P}_N){\cal L}_1\begin{bmatrix}
a\\
 \varphi
   \end{bmatrix}\right\|_p\\ 
  &\le& C N^{-q}\big(\|a\|_{p+q}+\|\varphi\|_{p+q}\big).\label{eq:proof:02a:theo:main02}
\end{eqnarray*}

\noindent 
{\bf Proof of \eqref{eq:01b:theo:main02} and \eqref{eq:03a:theo:main02}}
For ${\cal L}_2$, we proceed in the same fashion with 
\[                 
                   {\cal F}_{\pm}:=
\begin{bmatrix}
-{\rm K}_\pm&{\rm R}_\pm\\
-{\rm T}_\pm&{\rm K}^\top_\pm
\end{bmatrix},\quad \widetilde{\cal F}_{\pm,N}:=
\begin{bmatrix}
-\widetilde{\rm K}_{\pm,N}&\widetilde{\rm R}_{\pm,N}\\
-{\rm T}_{\pm,N}& \widetilde{\rm K}^\top_{\pm,N}.
\end{bmatrix}    
  \]
  which allows us to write
  \begin{eqnarray*}
   {\cal L}_2&=&(1+\nu^{-1})\begin{bmatrix}
                 &\Lambda\\
                 -\nu\,\mathrm{D\Lambda D}                 
                \end{bmatrix}+\begin{bmatrix}
                               \nu^{-1/2}&\\
         &
         \nu^{1/2}
                           \end{bmatrix}{\cal F}_{+}
\begin{bmatrix}
         \nu^{1/2}&\\
         &
         \nu^{-1/2}
                           \end{bmatrix}
+{\cal F}_{-},\\
\widetilde{\cal L}_{2,N}
                             &=&(1+\nu^{-1})\begin{bmatrix}
                 &\Lambda\\
                 -\nu\,\mathrm{D\Lambda D}                 
                \end{bmatrix}+
                \begin{bmatrix}
                               \nu^{-1/2}&\\
         &
         \nu^{1/2}
                           \end{bmatrix}{\cal F}_{+,N}
\begin{bmatrix}
         \nu^{1/2}&\\
         &
         \nu^{-1/2}
                           \end{bmatrix}
+{\cal F}_{-,N}.
  \end{eqnarray*}

Since ${\cal F}_\pm :H^{p+1}\times H^{p}\to H^{p+3}\times H^{p+2}$ holds as well, estimate 
\eqref{eq:interp_error} yields 
\[
 \left\|({\cal P}_N-{\cal I}){\cal F}_\pm \right\|_{H^{p+q+1}\times H^{p+q}\to H^{p+1}\times H^{p}}
                      \le C N^{-q-2}
 \]
for any $p\ge 0$ and $q\ge -1$. On the other hand, from \eqref{eq:03:prop:main}  (see
 also \eqref{eq:03a:prop:main}), 
\[
\|{\cal F}_k-\widetilde{\cal F}_{k,N}\|_{H^{p+q+1}\times H^{p+q}\to H^{p+1}\times H^p}\le C N^{-q-\min\{2,p\}}
\]
for any $p>1/2$ and $q\ge -2$  with $p+q>1/2$.

Thus
\begin{eqnarray}
\|{\cal L}_2-\widetilde{\cal L}_{2,N}\|_{H^{p+q+1}\times H^{p+q}\to H^{p+1}\times H^p }&&    \nonumber
\\
&&\hspace{-3.5cm}\le \max\{\nu,1\}
\bigg[
\|({\cal P}_N-{\cal I}) {\cal F}_{\pm } \|_{H^{p+q+1}\times H^{p+q}\to H^{p+1}\times H^p }
 \nonumber \\
%+
%\|({\cal P}_N-{\cal I}) {\cal E}_{-} \|_{H^{p+q}\times H^{p+q}\to H^p\times H^p }\\
&&\hspace{-3.2cm} +
\|{\cal P}_N\|_{H^{p+1}\times H^{p}\to H^{p+1}\times H^p } \nonumber
%\times\\
%&&\hspace{-4.6cm}
\big[\|
 {\cal F}_{\pm }-\widetilde{\cal F}_{\pm ,N}\|_{H^{p+q+1}\times H^{p+q}\to H^{p+1}\times H^p }
 %+\| {\cal D}_{-}-{\cal D}_{+,N}\|_{H^{p+q}\times H^{p+q}\to H^p\times H^p }
\big] \nonumber \\
&&\hspace{-3.5cm}\le  CN^{-q-\min\{2,p\}}\label{eq:L2:00}
\end{eqnarray}
which, with $q=0$, implies in particular \eqref{eq:01b:theo:main02}. Estimate \eqref{eq:03a:theo:main02} is proved from \eqref{eq:L2:00} as  in \eqref{eq:02a:theo:main02}.

\noindent{\bf Proof of \eqref{eq:01d:theo:main02} and \eqref{eq:03b:theo:main02}}. 
Notice first that ${\cal F}_{\pm}:H^{p}\times H^{p}\to H^{p+3}\times H^{p+1}$ is continuous and  
\[
\|
 {\cal F}_{\pm }-\widetilde{\cal F}_{\pm ,N}\|_{H^{p+q}\times H^{p+q}\to H^{p+2}\times H^p }\le C N^{-q-\min\{p,1\}},\quad p,p+q>1/2, \ q\ge -1
\]
which can be deduced from equations \eqref{eq:a:est:for:Fk} and \eqref{eq:c:est:for:Fk} with $p'=p+2$ and $q'=q-2$ and from equations \eqref{eq:b:est:for:Fk} and \eqref{eq:d:est:for:Fk}
with $p'=p$ and $q'=q$. Thus, similar arguments as those used above for ${\cal L}_2$ can be applied to show a different estimate: 
\begin{eqnarray}
\|{\cal L}_2-\widetilde{\cal L}_{2,N}\|_{H^{p+q+1}\times H^{p+q+1}\to H^{p+2}\times H^p }&& \nonumber
\\
&&\hspace{-3.5cm}
\le C \Big[
\|({\cal P}_N-{\cal I}) {\cal F}_{\pm } \|_{H^{p+q+1}\times H^{p+q+1}\to H^{p+2}\times H^p }
\nonumber\\
%+
%\|({\cal P}_N-{\cal I}) {\cal E}_{-} \|_{H^{p+q}\times H^{p+q}\to H^p\times H^p }\\
&&\hspace{-3.2cm} +
\|{\cal P}_N\|_{H^{p+2}\times H^{p}\to H^{p+2}\times H^p }
%\times\\
%&&\hspace{-4.6cm}
\big[\|
 {\cal F}_{\pm }-\widetilde{\cal F}_{\pm ,N}\|_{H^{p+q}\times H^{p+q}\to H^{p+2}\times H^p }
 %+\| {\cal D}_{-}-{\cal D}_{+,N}\|_{H^{p+q}\times H^{p+q}\to H^p\times H^p }
\Big]\nonumber\\
&&\hspace{-3.5cm}\le  C' N^{-q-2} \| {\cal F}_{\pm } \|_{H^{p+q+1}\times H^{p+q+1}\to H^{p+q+4}\times H^{p+q+2}}+C'  N^{-q-\min\{p,1\}}\nonumber \\
&&\hspace{-3.5cm}\le 
C''N^{-q-\min\{p,1\}} \label{eq:usedLater}
\end{eqnarray}
which holds for  $p>1/2$, $q\ge -1$ and $p+q>1/2$.

We are now ready to start analyzing the more complex formulation of this paper, namely ${\cal L}_3$ and the corresponding discretization given by $\widetilde{\cal L}_{3,N}$.  
Clearly,
\begin{eqnarray}
 {\cal L}_3-\widetilde{\cal L}_{3,N}&=&\frac{1}{\nu+1} ({\cal L}_1- \widetilde{\cal L}_{1,N})+
 \frac{2}{\nu+1}\begin{bmatrix}
      &{\rm R}_{\kappa}-\widetilde{\rm R}_{\kappa,N}\nonumber\\
      -\nu ({\rm T}_{\kappa}-{\rm T}_{\kappa,N})
      \end{bmatrix}{\cal L}_2
\\
  &&
      +\frac{2}{\nu+1}\begin{bmatrix}
      &{\rm V}_{\kappa,N}\nonumber\\
      -\nu {\rm H}_{\kappa,N}
      \end{bmatrix}({\cal L}_2-\widetilde{\cal L}_{2,N})\\
      &=:&{\cal T}_1+{\cal T}_2+{\cal T}_3.\label{eq:L3:02}
           \end{eqnarray}
First term with $\widetilde{\cal L}_{1,N}$ defined as ${\cal L}_{1,N}$ with $\widetilde{\rm V}_{\pm,N}$,$\widetilde{\rm K}_{\pm,N}$ and $\widetilde{\rm K}^\top_{\pm,N}$ instead,    can be analyzed as in \eqref{eq:L1:00} to get
\begin{equation}\label{eq:L3:025}
 \|{\cal T}_1\|_{H^{p+q+1}\times H^{p+q+1}\to H^{p+1}\times H^{p+1}}\le 
 CN^{-1}.
\end{equation}
For the second term we emphasize that
\begin{eqnarray}
\|{\cal T}_2\|_{H^{p+q+1}\times H^{p+q+1}\to H^{p+1}\times H^{p+1}}&&\nonumber\\
&& \hspace{-2cm} \le \:
CN^{-q-\min\{p,1\}}\|{\cal L}_2\|_{H^{p+q+1}\times H^{p+q+1}\to H^{p+q+1}\times H^{p+q}}.\qquad\label{eq:L3:03}
\end{eqnarray}
(We have applied \eqref{eq:c:est:for:Fk} with $p'=p+1$ and $q'=q-1$ and \eqref{eq:d:est:for:Fk}
with $p'=p+1$ and $q'=q$ and the mapping properties of ${\cal L}_2$). 

Regarding the third term, using \eqref{eq:usedLater} we get
\begin{eqnarray}
 \|{\cal T}_3\|_{H^{p+q+1}\times H^{p+q+1}\to H^{p+1}\times H^{p+1}}&\le&\nonumber
C\|{\cal L}_2-\widetilde{\cal L}_{2,N}\|_{H^{p+q+1}\times H^{p+q+1}\to   H^{p+2}\times H^{p}}
\\
&\le& C  N^{-q-\min\{p,1\}}.\label{eq:L3:04}
\end{eqnarray}
Gathering \eqref{eq:L3:025},   \eqref{eq:L3:03} and \eqref{eq:L3:04}   in \eqref{eq:L3:02}
we obtain
\begin{equation}\label{eq:L3:error}
\|{\cal L}_3-\widetilde{\cal L}_{3,N}\|_{H^{p+q+1}\times H^{p+q+1}\to H^{p+1}\times H^{p+1}}\le C N^{-q-\min\{p,1\}}
 \end{equation}
 which implies \eqref{eq:01d:theo:main02} by taking $q=0$.

To prove \eqref{eq:03b:theo:main02}, we can easily see that, as in \eqref{eq:proof:02a:theo:main02},  we simply have to bound
\[
  \left\|({\cal L}_{3}-\widetilde{\cal L}_{3,N})\begin{bmatrix}
a\\
   \varphi
   \end{bmatrix}\right\|_{p+1},\quad 
   \left\|\big({\cal P}_N{\cal R}_{\kappa,N}-{\cal R}_{\kappa}\big)\begin{bmatrix}
f\\
 \lambda\end{bmatrix}  \right\|_{p+1}
\]
The first term has been already studied in \eqref{eq:L3:error}. Regarding the second term, we have
\begin{eqnarray*}
 \left\|\big({\cal P}_N{\cal R}_{\kappa,N}-{\cal R}_{\kappa}\big)\begin{bmatrix}
f\\
 \lambda\end{bmatrix}  \right\|_{p+1}&&\\
 &&\hspace{-3.2cm}\le C\big[ \|{\cal R}_{\kappa,N}-{\cal R}_{\kappa}\|_{H^{p+q+1}\times H^{p+q+1}\to H^{p+1}\times H^{p+1}}\\
 &&\hspace{-2.7cm} 
 +\|({\cal P}_N-{\cal I}) {\cal R}_{\kappa}\|_{H^{p+q+1}\times H^{p+q+1}\to H^{p+1}\times H^{p+1}}\big]\big[\|f\|_{p+q+1}+\|\lambda\|_{p+q+1} \big]\\
  &&\hspace{-3.2cm}\le C N^{-q}\big[ \|f\|_{p+q+1}+\|\lambda\|_{p+q+1}\big].
\end{eqnarray*}
Notice that, unlike \eqref{eq:02a:theo:main02}, $ \|f\|_{p+q+1}$, $\|\lambda\|_{p+q+1}$  cannot be bounded  in terms of $\|a\|_{p+q+1}$ and $\|b\|_{p+q+1}$ because we cannot guarantee  that ${\cal R}_\kappa$  is invertible. However, it follows that
\begin{eqnarray*}
 \|a\|_{p+q+1}+\|b\|_{p+q+1}&&\\
 && \hspace{-2cm}\le\|{\cal L}^{-1}_{3}{\cal R}_{\kappa }\|_{H^{p+q+1}\times H^{p+q+1}\to
 H^{p+q+1}\times H^{p+q+1}}\big[\|f\|_{p+q+1}+\lambda\|_{p+q+1}\big]
\end{eqnarray*}
which allows us to write the convergence in terms of the regularity of the right-hand-side instead.

\end{proof}
The main point of this theorem is that convergence in higher Sobolev space norms of the Helmholtz boundary operators  allows to prove easily the stability and convergence of the Nystr\"om discretizations. The higher order discretizations $\{\widetilde{\rm V}_{k,N},$ $ \widetilde{\rm R}_{K,N},\widetilde{\rm R}^\top_{K,N}\}$ guarantee convergence of Nystr\"om discretizations for rather complex formulations whereas the simpler, but less accurate discretizations of second kind integral formulations such as those based on the operators ${\cal L}_1$ still converge. The analysis based on the results of Theorem \ref{theo:01}, whose details are a bit more subtle,  allows us to employ optimal discretizations and norms in which the stability and convergence results hold. Observe that on account of Sobolev embedding theorems, all of the convergence results established above imply convergence in the $L^\infty$ norm.

% The first result in this theorem is just the well-posedness of the formulations considered in this section. Second set of inequalities implies in particular the stability of the discretizations. Observe that 
% ${\cal L}_3$ is stable in both $H^{p}\times H^{p+1}$ and $H^{p}\times H^{p}$, so that we can state the equation 
% in both product spaces. Proving the latter functional result is more subtle and makes use of Calderon identities and extra regularizing properties of the difference of the same layer operators for distinct wave-numbers (see \cite{CVY} for a more detailed analysis). As we have shown in this theorem,  the higher accurate discretization \eqref{eq:L3Ntilde} is needed to 
% define a stable and convergence method in the stronger space. 

\section{Numerical experiments}
\begin{figure}[tb]
%\sidecaption
% For example, with the graphicx style use
\includegraphics[width=.45\textwidth]{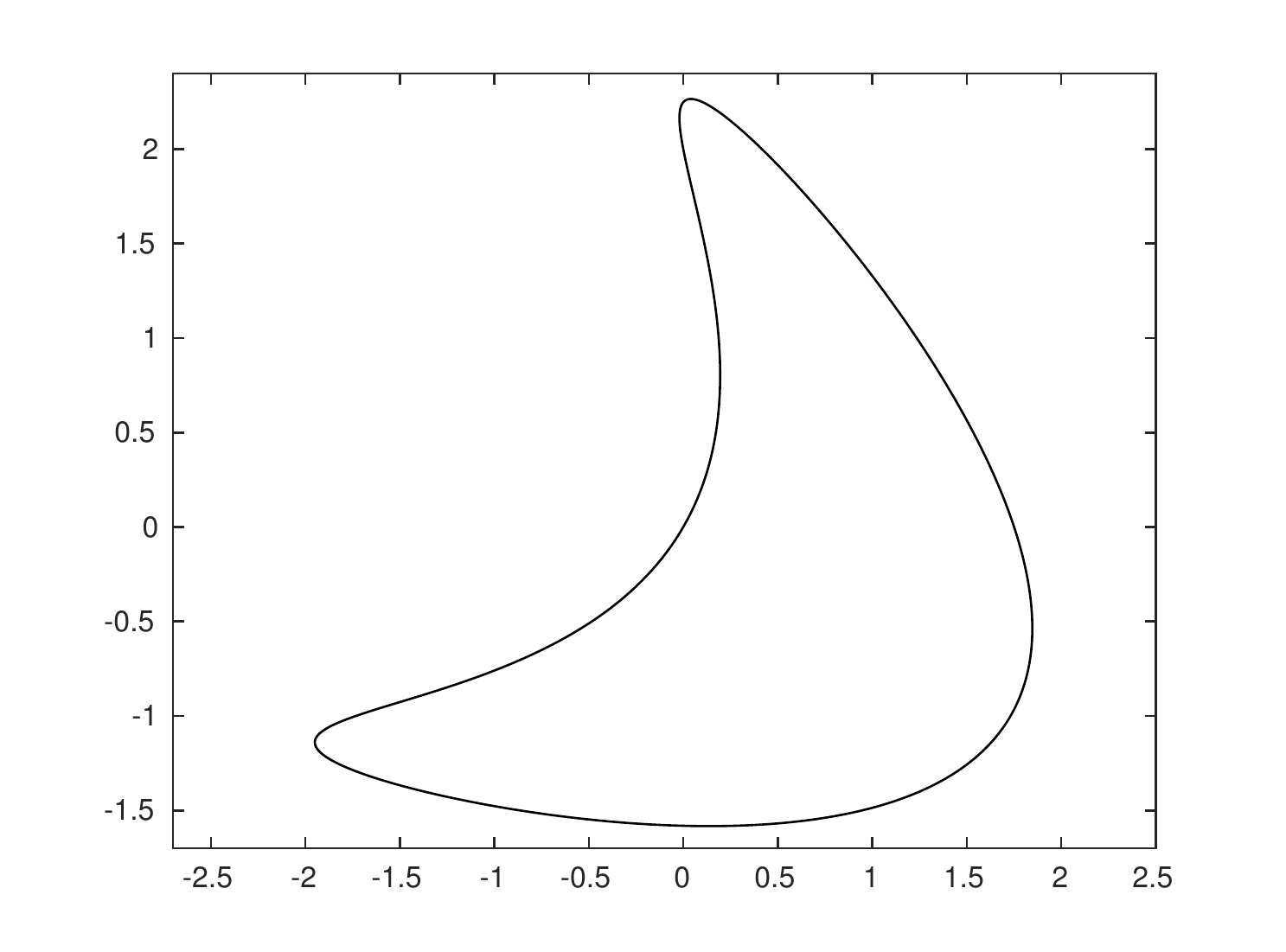}\quad
\includegraphics[width=.45\textwidth]{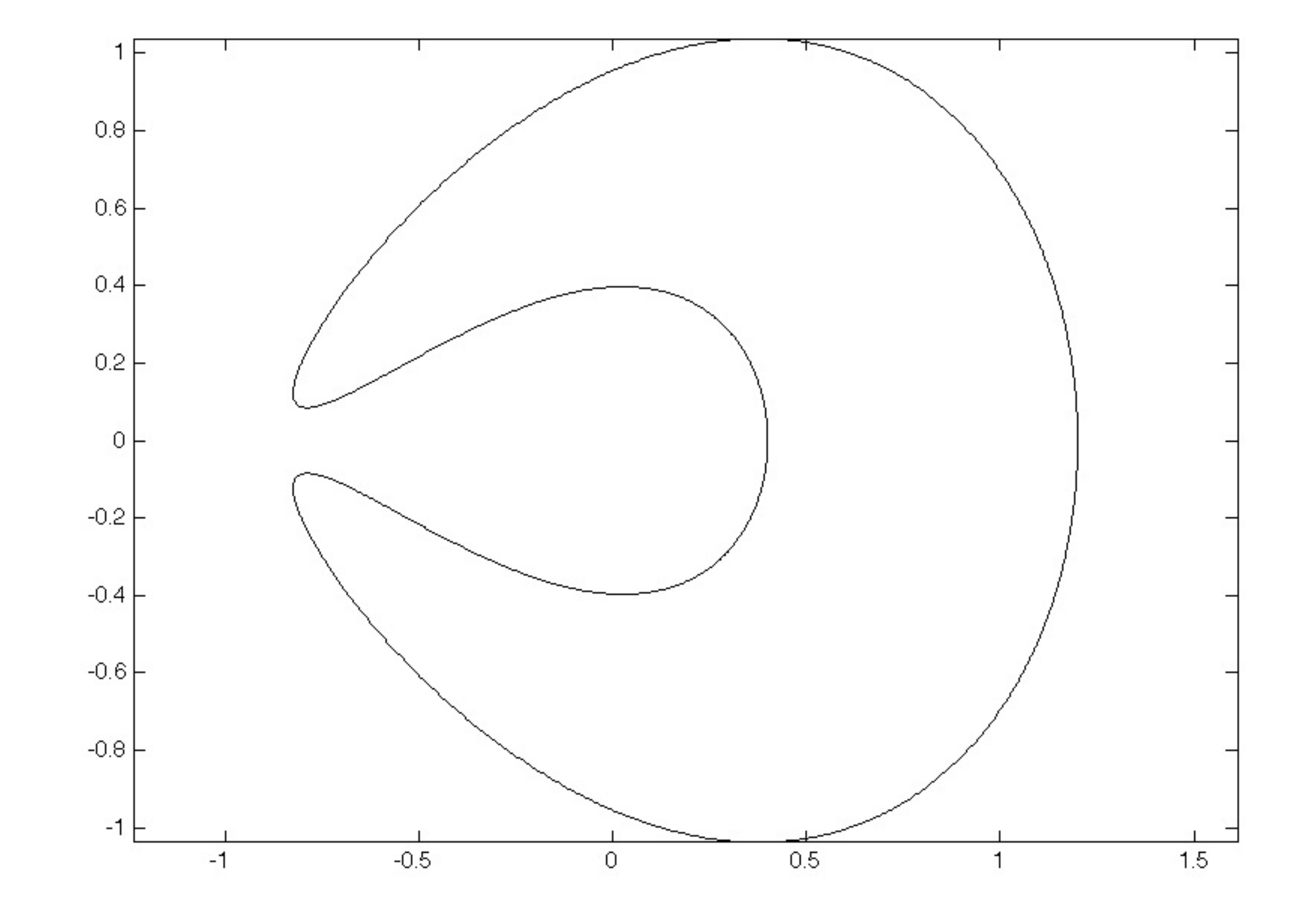}
%
% If no graphics program available, insert a blank space i.e. use
%\picplace{5cm}{2cm} % Give the correct figure height and width in cm
%
\label{fig:cavity}
\caption{Kite and cavity geometries considered in the numerical experiments } 
\end{figure}

For brevity,  we only present numerical results for the Costabel-Stephan formulation ${\cal L}_2$. We refer the reader to 
\cite{CVY,CMV} for extensive numerical results for the other formulations.

\begin{table}[h]
\[
\begin{array}{c|c|c|c|c}
  & \multicolumn{2}{c|}{\rm Kite} & \multicolumn{2}{c}{\rm Cavity} \\
 \hline&&&\\[-2ex]
 N   & {{\cal L}}_{2,N}      &    \widetilde{{\cal L}}_{2,N}
  & {{\cal L}}_{2,N}      &    \widetilde{{\cal L}}_{2,N}\\[1ex]
 \hline&&&\\[-2ex]
96   &    3.2{\rm E}-02      &         9.1{\rm E}-03   &   7.1{\rm E}-02     &        1.7{\rm E}-02    \\[1ex]
128  &    4.1{\rm E}-04      &         2.5{\rm E}-05   & 
   8.3{\rm E}-04     &        6.3{\rm E}-05  \\[1ex]
160  &    5.9{\rm E}-11      &          5.8{\rm E}-12 &
2.0{\rm E}-10     &        3.3{\rm E}-11
\end{array}
\]
\caption{\label{tab:kite} $L^\infty$ error estimate in the far field for the discretizations ${{\cal L}}_{2,N}$ and $\widetilde{{\cal L}}_{2,N}$ for the Helmholtz transmission problem in the kite (left) and cavity (right) domains.} 
\end{table}

% \pgfdeclareimage[width=.5\linewidth]{Domain2}{cavity.pdf}
% \pgfdeclareimage[width=.5\linewidth]{Domain}{kite.pdf}
% \begin{figure}[H]
%  \begin{center}
%   \begin{tikzpicture}[font=\footnotesize]
%     \pgftext[at=\pgfpoint{3mm}{-1mm},left,base]{\pgfuseimage{Domain}\pgfuseimage{Domain2}}
%     %\draw (4.5, 65mm) node {\large $\Omega_1$};
%     %\draw (4.5, 37mm) node {\large $\Omega_2$};
%     %\draw (6.2, 73mm) node {\large \color{red}{$\Gamma_0$}};
%     %\draw (7.4, 11mm) node {\large \color{blue}{$\Gamma_1$}};
%   \end{tikzpicture}  
%   %\begin{tikzpicture}[font=\footnotesize]
%   %  \pgftext[at=\pgfpoint{3mm}{-1mm},left,base]{\pgfuseimage{Domain2}}
%     %\draw (4.5, 65mm) node {\large $\Omega_1$};
%     %\draw (4.5, 37mm) node {\large $\Omega_2$};
%     %\draw (6.2, 73mm) node {\large \color{red}{$\Gamma_0$}};
%     %\draw (7.4, 11mm) node {\large \color{blue}{$\Gamma_1$}};
% %  \end{tikzpicture}
% \caption{\label{fig:cavity}Kite and cavity geometries considered in the numerical experiments  }
%  \end{center}
% \end{figure}

The domains we have considered are the geometries depicted in Figure \ref{fig:cavity}. 
We have taken $k_+ = 8, k_-=32$ in the Helmholtz transmission problems \eqref{eq:Tr}, with $\nu=1$ in the transmission conditions across the interface. We have applied the numerical schemes $\widetilde{\cal L}_{2,N}$ and ${\cal L}_{2,N}$. The latter   scheme is that defined using ${\rm V}_{\pm,N}$, the less accurate approximation for ${\rm V}_\pm$. We point out that only the first discretization has been analyzed in this paper.

The $L^\infty$ error estimate  in the far field for the numerical solutions is shown in Table \ref{tab:kite}.  
 The {\em exact} solution has been computed using ${\cal L}_{1,N}$ for sufficiently large  $N$, which, in turns, provides an indirect demonstration of the performance of this discretization too.

Both  methods converge superalgebraically to the exact solution, although $\widetilde{\cal L}_{2,N}$ performs better with even a slightly faster convergence. Convergence, and specially  stability  of ${\cal L}_{2,N}$ remains as an open problem and certainly will deserve more research in the future. 

\subsection*{Acknowledgments}
Catalin Turc gratefully acknowledge support from NSF through contract DMS-1312169. V\'{\i}ctor
Dom\'{\i}nguez is partially supported by Ministerio de Econom\'{\i}a y Competitividad, through the grant
MTM2014-52859.


\begin{thebibliography}{99} 





\bibitem{CVY} 
Y.~Boubendir, V.~Dom\'{\i}nguez,   and C.~Turc.
{\em High-order Nystr\"om discretizations for the solution of integral equation 
formulations of two-dimensional Helmholtz transmission problems.} 
\newblock  To appear in 
IMA J. Numer. Anal. 



\bibitem{CoSt}
M.~Costabel and E.~Stephan.
{\em A direct boundary integral equation method for transmission problems.}
 { J. Math. Anal. Appl.} {\bf 106} (1985), no 2, 367--413.

 


\bibitem{CMV} V. Dom\'{\i}nguez, M. Lyon, and C. Turc. 
{\em High-order Nystr\"om discretizations for the solution of integral equation formulations of two-dimensional Helmholtz transmission on interfaces with corners.}
\newblock Submitted. Preprint available in {\em arXiv:1509.04415} 
% 
% \bibitem{KittapaKleinman} 
% R.~Kittappa and R.~E. Kleinman.
% {\em Acoustic scattering by penetrable homogeneous objects.}
%  {J. Mathematical Phys.}  {\bf 16} (1975), 421--432.



\bibitem{KlMa} 
R.~E. Kleinman and P.~A. Martin.
{\em On single integral equations for the transmission problem of
  acoustics.}
 {SIAM J. Appl. Math.}, \textbf{48} (1988), no 2, 307--325.

\bibitem{kressRoach}  R. Kress and G. F. Roach. {\em Transmission problems for the Helmholtz equation.} J. Mathematical
Phys., 19(6):1433--1437, 1978.



\bibitem{kressH} 
R.~Kress.
{\em On the numerical solution of a hypersingular integral equation in
  scattering theory.}
{ J. Comput. Appl. Math.} {\bf 61} (1995), no 3, 345--360.

\bibitem{KressLI} 
R.~Kress.
{Linear integral equations}. 
%Volume~82 of {\em Applied Mathematical Sciences}.
Springer-Verlag, New York, third edition, 2014.

\bibitem{kusmaul} R. Kussmaul. {\em Ein numerisches Verfahren zur Losung des Neumannschen
 Aussenraumproblems fur die Helmholtzsche Schwingungsgleichung}. Computing {\bf 4} (1969), 246--273.
  

 
\bibitem{Maue} A.W. Maue.{\em Zur Formulierung eines allgemeinen Beugungsproblems durch eine Integralgleichung.}
Z. Phys. {\bf 126} (1949), 601--618 

\bibitem{martensen} E. Martensen. {\em Uber eine Methode zum raumlichen Neumannschen Problem mit einer ANwendung fur torusartige Berandungen}, Acta Math. {\bf 109}  (1963), 75--135.

\bibitem{Nedelec}  J.-C N\'{e}d\'{e}lec. 
{\em Integral equations with
nonintegrable kernels.}
  Integral Equations Operator Theory {\bf 5} (1982), no 4, 562--572.
  
  \bibitem{Nedelec02}  J.-C N\'{e}d\'{e}lec and J. Planchard.
{\em Une m\'ethode variationnelle d'\'el\'ements finis pour la
              r\'esolution num\'erique d'un probl\`eme ext\'erieur dans
              {$\R^{3}$.}}
  Rev. Fran\c caise Automat. Informat. Recherche
              Op\'erationnelle S\'er. Rouge {\bf 7} (1973),  105-129.
 
\bibitem{saranen_vainiko} 
J.~Saranen and G.~Vainikko.
{\em Periodic integral and pseudodifferential equations with
  numerical approximation}.
%\newblock Springer Monographs in Mathematics. 
Springer-Verlag, Berlin, 2002.
\end{thebibliography}
\end{document}